\documentclass{amsart}

\usepackage{amssymb,amsmath}
\usepackage[usenames, dvipsnames]{color}

\newtheorem{theorem}{Theorem}[section]
\newtheorem{lemma}[theorem]{Lemma}
\newtheorem{corollary}[theorem]{Corollary}
\newtheorem{proposition}[theorem]{Proposition}

\theoremstyle{definition}
\newtheorem{definition}[theorem]{Definition}

\theoremstyle{remark}
\newtheorem{remark}[theorem]{Remark}

\numberwithin{equation}{section}

\newenvironment{PfOfThmPressure}[1]
{\par\vskip2\parsep\noindent{\sc Proof of Theorem\ \ref{Thm:Pressure}. }}{{\hfill
$\Box$}
\par\vskip2\parsep}

\newenvironment{PfOfThmEscapeRate}[1]
{\par\vskip2\parsep\noindent{\sc Proof of Theorem\ \ref{Thm:EscapeRate}. }}{{\hfill
$\Box$}
\par\vskip2\parsep}

\def\R{\mathbb{R}}
\def\N{\mathbb{N}}

\DeclareMathOperator{\Var}{Var}

\begin{document}

\title[Pressure for random SFTs]{Pressure and escape rates for random subshifts of finite type}


\author{Kevin McGoff}
\address{9201 University City Blvd. \\
Charlotte, NC 28223}
\curraddr{}
\email{kmcgoff1@uncc.edu}
\thanks{}



\date{}

\begin{abstract}
In this work we consider several aspects of the thermodynamic formalism in a randomized setting. Let $X$ be a non-trivial mixing shift of finite type, and let $f : X \to \R$ be a H\"{o}lder continuous potential with associated Gibbs measure $\mu$. Further, fix a parameter $\alpha \in (0,1)$. For each $n \geq 1$, let $\mathcal{F}_n$ be a random subset of words of length $n$, where each word of length $n$ that appears in $X$ is included in $\mathcal{F}_n$ with probability $1-\alpha$ (and excluded with probability $\alpha$), independently of all other words. Then let $Y_n = Y(\mathcal{F}_n)$ be the random subshift of finite type obtained by forbidding the words in $\mathcal{F}_n$ from $X$. In our first main result, for $\alpha$ sufficiently close to $1$ and $n$ tending to infinity, we show that the pressure of $f$ on $Y_n$ converges in probability to the value $P_X(f) + \log(\alpha)$, where $P_X(f)$ is the pressure of $f$ on $X$.  Additionally, let $H_n = H(\mathcal{F}_n)$ be the random hole in $X$ consisting of the union of the cylinder sets of the words in $\mathcal{F}_n$.  For our second main result, for $\alpha$ sufficiently close to one and $n$ tending to infinity, we show that the escape rate of $\mu$-mass through $H_n$ converges in probability to the value $-\log(\alpha)$ as $n$ tends to infinity. 
\end{abstract}


\maketitle


\section{Introduction}

Random subshifts of finite type were introduced in \cite{McGoff2012}, and they have subsequently been studied in \cite{Broderick2017, McGoffPavlov2018, McGoffPavlov2016}. Let us quickly recall their definition. Let $X$ be a non-trivial mixing subshift of finite type (SFT), and let $B_n(X)$ be the set of words of length $n$ that appear in $X$. For a fixed parameter $\alpha \in (0,1)$ and $n \in \N$, let $\mathcal{F}_{n}$ be the randomly selected subset of $B_n(X)$ formed by including each word from $B_n(X)$ in $\mathcal{F}_n$ with probability $1-\alpha$ (and excluding it from $\mathcal{F}_n$ with probability $\alpha$), independently of all other words. Then let $Y_n = Y(\mathcal{F}_n)$ be the set of points in $X$ that do not contain any word from $\mathcal{F}_n$. We refer to $Y_n$ as a random SFT.
In order to study random SFTs, we fix the ambient system $X$ and the parameter $\alpha$. Then we seek to describe the properties of $Y_n$ that have probability tending to one as $n$ tends to infinity. This framework gives a precise way to describe the behavior of ``typical" SFTs within the ambient system $X$.

Previous work on random SFTs \cite{Broderick2017, McGoff2012, McGoffPavlov2018, McGoffPavlov2016} has established the existence of at least one critical value $\alpha_c$ such that the typical behavior of $Y_n$ changes abruptly as $\alpha$ crosses this value. Indeed, when $\alpha < \alpha_c$, there is a positive limit for the probability that $Y_n$ is empty, and  $Y_n$ has zero entropy with probability tending to one. On the other hand, for $\alpha > \alpha_c$, the probability that $Y_n$ is empty tends to zero, and the entropy of $Y_n$ converges in probability to the value $h(X) + \log(\alpha)$, where $h(X)$ is the entropy of $X$.  (Note that this value is positive for $\alpha > \alpha_c$.) Furthermore, for $\alpha$ close enough to one, $Y_n$ contains a unique ``giant component," which is itself a mixing SFT with full entropy, as well as a random number of isolated periodic orbits. See \cite{McGoff2012} for details.

In the present work, we study some aspects of the thermodynamic formalism for random SFTs in the super-critical regime ($\alpha > \alpha_c$). In our first main result (Theorem \ref{Thm:Pressure}), we describe the distribution of the pressure of random SFTs for a fixed potential function, and in our second main result (Theorem \ref{Thm:EscapeRate}), we describe the distribution of the escape rate of mass of Gibbs measures through random holes. Although these two topics may not at first appear to be related, they are in fact quite closely connected, as demonstrated by Proposition \ref{Prop:Crayola}.

\subsection{Pressure of random SFTs}

 Suppose that $f : X \to \R$ is a fixed H\"{o}lder continuous potential function. We seek to identify the limiting behavior of the pressure of $f$ restricted to the random SFT $Y_n$ in the limit as $n$ tends to infinity. For notation, let $P_Y(f)$ denote the topological pressure of $f$ restricted to any subshift $Y \subset X$ (see Section \ref{Sect:Background} for definitions). 

%

\begin{theorem} \label{Thm:Pressure}
Let $X$ be a non-trivial mixing SFT and $f : X \to \R$ H\"{o}lder continuous. Then there exists $\gamma_0 \in (0,1)$ such that for each $\alpha \in (\gamma_0,1]$ and for each $\epsilon >0$, there exists $\rho >0$ such that for all large enough $n$,
\begin{equation*}
 \mathbb{P}_{\alpha}\biggl( \bigl| P_{Y_n}(f) - \bigl( P_X(f) + \log( \alpha) \bigr) \bigr| \geq \epsilon \biggr) < e^{-\rho n}.
\end{equation*}
\end{theorem}
In other words, the pressure of $f$ on the random SFT $Y_n$ converges in probability to the value $P_X(f) + \log(\alpha)$. Note that when $f \equiv 0$, we recover the result of \cite{McGoff2012} regarding the entropy of random SFTs. 

\begin{remark}
The particular value of $\gamma_0$ that we use in our proof is given in Definition \ref{Def:Gamma}. It is possible that this value of $\gamma_0$ is not optimal, in the sense that for some choices of $X$ and $f$, the statement might remain true with a smaller value of $\gamma_0$. However, one may check that for $f \equiv 0$, our definition of $\gamma_0$ is equal to $\alpha_c$, which is optimal.
\end{remark}

\begin{remark}
The proof of Theorem \ref{Thm:Pressure} appears in Section \ref{Sect:MainProof}. The broad outline of this proof is similar to the outline of the proof of \cite[Theorem 1.3]{McGoff2012} concerning the entropy of random SFTs. However, the core technical results in the proof, which appear in Section \ref{Sect:RepeatProofs}, require new ideas to handle the fact that $f$ may not be zero. In particular, we must estimate the $\mu$-measure of the appearance of certain types of repeated patterns, where $\mu$ is a Gibbs measure but not necessarily a measure of maximal entropy, and this generalization requires substantially new ideas.
\end{remark}

\subsection{Escape rate through random holes}

Our second main result involves thinking of the random set of forbidden words $\mathcal{F}_n$ as a hole in the ambient system $X$, creating an open dynamical system. For an introduction to open systems, see \cite{DemersYoung} and references therein. Previous work on open systems has focused largely on the existence and properties of the escape rate of mass through the hole, as well as the existence and properties of conditionally invariant distributions (conditional on avoiding the hole); for an incomplete sampling of the literature on open systems, see
\cite{BDM2010,Bunimovich2011,CM2,CM1,CMT1,CMT2,CvdB,CMS1994,CMS1996,CMS1997,DemersFernandez2016,DWY2010,D2005a,D2005b,Demers2014,DIMMY2017,DemersWright2012,DWY2012,FP2012,FroylandStancevic2010,Keller2012,LMD,LopesM,PianigianiYorke1979}.
In this work we focus on the escape rate of mass through the hole, which we define below.
Let $\sigma : X \to X$ denote the left-shift map on $X$. For a Borel probability measure $\mu$ on $X$ and a hole $H \subset X$, the escape rate of $\mu$ through $H$ is defined to be $- \varrho(\mu : H)$, where
\begin{equation*}
\varrho(\mu : H) = \lim_m \frac{1}{m} \log \mu \biggl( \bigl\{ x \in X : \forall k \in \{0,\dots,m-1\}, \, \sigma^k(x) \notin H \bigr\} \biggr),
\end{equation*} 
whenever the limit exists. 

 In this work we consider random holes $H_n$, constructed from the set of forbidden words $\mathcal{F}_n$ as follows:
\begin{equation*}
H_n = \bigcup_{w \in \mathcal{F}_n} \bigl\{ x \in X : x_0 \dots x_{n-1} = w \bigr\}.
\end{equation*}
Our goal is to describe the escape rate of mass of Gibbs measures through $H_n$. 
\begin{theorem} \label{Thm:EscapeRate}
Let $X$ be a non-trivial mixing SFT and let $\mu$ be the Gibbs measure associated to the H\"{o}lder continuous potential function $f : X \to \R$. Then there exists $\gamma_0 \in (0,1)$ such that for each $\alpha \in (\gamma_0,1]$ and for each $\epsilon >0$, there exists $\rho >0$ such that for all large enough $n$,
\begin{equation*}
 \mathbb{P} \biggl( \bigl| \varrho(\mu : H_n) - \log(\alpha) \bigr| \geq \epsilon \biggr) < e^{-\rho n}.
\end{equation*}
\end{theorem}
Thus the escape rate of $\mu$ through the randomly selected hole $H_n$ converges in probability to $-\log(\alpha)$. 

\begin{remark} \label{Rmk:BigHoles}
With probability tending to one, the hole $H_n$ consists of the union of approximately $(1-\alpha) |B_n(X)|$ cylinder sets of length $n$. One may think of $H_n$ as typically consisting of the union of many small holes spread randomly throughout the state space. Furthermore, the expected value of the $\mu$-measure of $H_n$ in $1-\alpha$, which remains bounded away from zero as $n$ tends to infinity. In this sense, the holes considered here differ substantially from the ``small holes" studied in some previous work \cite{Bunimovich2011,DWY2010,Demers2014,FP2012}.
\end{remark}

\begin{remark}
In \cite{Bunimovich2011}, the authors prove that in the deterministic setting, the escape rate depends on both the size (measure) of the hole and its precise location in state space. In contrast, Theorem \ref{Thm:EscapeRate} shows that for random holes, the escape rate is well approximated by a function that depends only on the expected measure of the hole. Indeed, the expected measure of the hole is $(1-\alpha)$ (see Remark \ref{Rmk:Measure}), so the expected measure of its complement is $\alpha$. Then Theorem \ref{Thm:EscapeRate} yields that the escape rate converges in probability to the value $-\log(\alpha)$.
\end{remark}

\begin{remark}
A naive approximation of the hitting time for the hole is given by a geometrically distributed random variable $\tau$ with probability of success $p = 1-\alpha$. Since $\mathbb{P}(\tau = k) = \alpha^{k-1} (1-\alpha)$, we have that
\begin{equation*}
\lim_k \frac{1}{k} \log \mathbb{P}(\tau > k) = \log \alpha.
\end{equation*}
From this perspective, Theorem \ref{Thm:EscapeRate} may be interpreted as giving precise meaning to the statement that the escape rate through the random hole is approximately the same what one would obtain if the hitting time of the hole were geometrically distributed with probability of success equal to the measure of the hole.
\end{remark}

As a consequence of Theorem \ref{Thm:EscapeRate}, we can also estimate the escape rate of mass of Gibbs measures for Axiom A diffeomorphisms through randomly selected Markov holes. As the proof relies solely on Theorem \ref{Thm:EscapeRate} and the well-known relationship between Markov partitions for Axiom A diffeomorphisms and SFTs (see \cite{Bowen1975}), we omit the proof.
\begin{corollary} \label{Cor:AxiomA}
Let $T : M \to M$ be an Axiom A diffeomorphism such that the restriction of $T$ to its non-wandering set is topologically mixing, and let $f : M \to \R$ be a H\"{o}lder continuous potential with Gibbs measure $\nu$. Let $\xi$ be a finite Markov partition of $M$ with diameter small enough that the symbolic dynamics is well-defined, and let $\xi^n = \vee_{k = 0}^{n-1} T^{-k} \xi$. Then there exists $\gamma_0 \in (0,1)$ such that for each $\alpha \in (\gamma_0,1]$, the following holds. Let $H_n$ be the randomly selected hole obtained by including each cell of $\xi^n$ independently with probability $1-\alpha$. Then for each $\epsilon >0$, there exists $\rho >0$ such that for large enough $n$,
\begin{equation*}
 \mathbb{P} \biggl( \bigl| \varrho(\nu : H_n) - \log(\alpha) \bigr| \geq \epsilon \biggr) < e^{-\rho n}.
\end{equation*}
\end{corollary}



\subsection{Outline of the paper}

The following section collects some background definitions and results that are used elsewhere in the paper.
In Section \ref{Sect:MainProof}, we present the proof Theorem \ref{Thm:Pressure} with the help of several technical lemmas.
These technical lemmas are then proved in Sections \ref{Sect:RepeatProofs}, \ref{Sect:MomentBounds}, and \ref{Sect:UBandLB}.
Finally, in Section \ref{Sect:Connections}, we establish Proposition \ref{Prop:Crayola}, which relates escape rates to pressure, and then we prove Theorem \ref{Thm:EscapeRate}.

\section{Background and notation} \label{Sect:Background}

\subsection{Symbolic dynamics}
Let $\mathcal{A}$ be a finite set, which we call the alphabet. We let $\Sigma = \mathcal{A}^{\mathbb{Z}}$ denote the full-shift, and we let $\sigma : \Sigma \to \Sigma$ be the left-shift map, $\sigma(x)_n = x_{n+1}$. We endow $\Sigma$ with the product topology from the discrete topology on $\mathcal{A}$, which makes $\sigma$ a homeomorphism. We define the metric $d( \cdot, \cdot)$ on $\Sigma$ by the rule $d(x,y) = 2^{-n(x,y)}$, where $n(x,y)$ is the infimum of all $|m|$ such that $x_m \neq y_m$.

A subset $X \subset \Sigma$ is a \textit{subshift} if it is closed and $\sigma(X) = X$. In the context of a subshift $X$, we also use the symbol $\sigma$ to denote the restriction of the left-shift to $X$. A \textit{word} on $\mathcal{A}$ is an element of $\mathcal{A}^m$ for some $m \geq 1$. We also refer to the empty word as a word. If $u = u_1 \dots u_m$ is in $\mathcal{A}^m$, then we say that $u$ has length $m$, and we let $u_i^j$ denote the subword $u_i \dots u_j$. Further, we let $B_m(X)$ denote the set of words of length $m$ that appear in some point in $X$. For any word $w$ in $B_m(X)$, we let $[w]$ denote the set of points $x \in X$ such that $x_0 \dots x_{m-1} = w$. Also, for $x \in X$ and $i \leq j$, we let $x[i,j]$ denote the set of points $y \in X$ such that $y_i \dots y_j = x_i \dots x_j$. 

A subset $X \subset \Sigma$ is a \textit{subshift of finite type} (SFT) if there exists a natural number $m$ and a collection of words $\mathcal{F} \subset \mathcal{A}^m$ such that $X$ is exactly the set of points in $\Sigma$ that contain no words from $\mathcal{F}$. We say that an SFT is non-trivial if it contains at least two points. The SFT $X$ is mixing if there exists $N$ such that for all points $x,y \in X$, there exists a point $z \in X$ such that $x(-\infty,0] = z(-\infty,0]$ and $y[N,\infty) = z[N,\infty)$.

For any subshift $X \subset \Sigma$, we let $M(X,\sigma)$ denote the set of Borel probability measures $\mu$ on $X$ such that $\mu(\sigma^{-1}A) = \mu(A)$ for all Borel sets $A \subset X$. Suppose $\mu \in M(X,\sigma)$. When it will not cause confusion, we write $\mu(w)$ to denote the measure of the cylinder set $[w]$, where $w \in B_m(X)$ for some $m \geq 1$.

For any measure $\mu \in M(X,\sigma)$, one may define the entropy of $\mu$ as
\begin{equation*}
h(\mu) = \lim_m \frac{1}{m} \sum_{w \in B_m(X)} - \mu(w) \log \mu(w),
\end{equation*}
where the limit exists by subadditivity.

\subsection{Pressure and equilibrium states} \label{Sect:BackgroundPressure}

Let $Y$ be a subshift, and let $f : Y \to \R$ be continuous. For $m \geq 1$ and $w$ in $B_m(Y)$, let
\begin{equation*}
S_m f(w) = \sup_{x \in Y \cap [w]} \sum_{k=0}^{m-1} f \circ \sigma^k (x).
\end{equation*}
Then let
\begin{equation*}
\Lambda_m(Y) = \sum_{w \in B_m(Y)} e^{S_m f(w)}.
\end{equation*}
Finally, the (topological) pressure of $f$ on $Y$ is defined as
\begin{equation*}
 P_Y(f) = \lim_{m \to \infty} \frac{1}{m} \log \Lambda_m(Y),
\end{equation*}
where the limit exists by subadditivity. 

The well-known Variational Principle (see \cite{W}) states that
\begin{equation*}
P_Y(f) = \sup \biggl\{ \int f \, d\mu + h(\mu) : \mu \in M(Y,\sigma) \biggr\}.
\end{equation*}
For a subshift $Y$, this supremum must be realized, and any measure that attains the supremum is known as an equilibrium state for $f$ on $Y$.


Now suppose that $X$ is a mixing SFT and $f: X \to \R$ is H\"{o}lder continuous. In this case, it is known that there is a unique equilibrium state $\mu \in M(X,\sigma)$ for $f$, and furthermore $\mu$ satisfies the following Gibbs property: there exists $K>1$ such that for all $n \geq 1$ and $x \in X$,
\begin{equation} \label{Eqn:GibbsProp}
K^{-1} \leq \frac{ \mu \bigl( x[0,n-1] \bigr) }{ \exp\bigl( -P_X(f) \cdot n + \sum_{k=0}^{n-1} f \circ \sigma^k(x) \bigr) } \leq K.
\end{equation}

We may now give a definition for the parameter $\gamma_0$ that appears in Theorems \ref{Thm:Pressure} and \ref{Thm:EscapeRate}. 
\begin{definition} \label{Def:Gamma}
Let $X$ be a non-empty mixing SFT, and let $f : X \to \R$ be a H\"{o}lder continuous potential with associated Gibbs measure $\mu$. 
Then let $\gamma_0 = \gamma_0(X,f)$ be defined by
\begin{equation*}
\gamma_0 = \inf\Bigl\{ \gamma > 0 : \exists n_0, \forall m \geq n_0, \forall u \in B_m(X), \, \mu(u) \leq \gamma^m \Bigr\}.
\end{equation*}
\end{definition}

Note that by \cite[Lemma 5]{Abadi}, if $X$ is non-trivial, then $\gamma_0 < 1$. We also make the following remark. Suppose $1 \geq \gamma > \gamma_0$, and fix $n_0$ such that $\mu(u) \leq \gamma^{|u|}$ whenever $|u| \geq n_0$. Then for any word $u$ we have $\mu(u) \leq \gamma^{|u| - n_0}$; indeed, if $|u| \geq n_0$, then it follows from the choice of $n_0$, and if $|u| \leq n_0$, then $\mu(u) \leq 1 \leq \gamma^{|u| - n_0}$. 

It is well-known (see, \textit{e.g.}, \cite[Proof of Proposition 1.14]{Bowen1975}) that $\mu$ satisfies a mixing property called $\psi$-mixing, from which a variety of mixing-type estimates may be deduced. The bounds required for the present work are summarized in the following lemma, which we state without proof.
\begin{lemma} \label{Lemma:Oscar}
Let $X$ be a non-trivial mixing SFT with H\"{o}lder continuous potential $f : X \to \R$ and associated Gibbs measure $\mu$. 
Then there exist constants $K >0$ and $g_0 \geq 1$ such that:
\begin{itemize}
\item the Gibbs property (\ref{Eqn:GibbsProp}) holds;
\item for all $m,n \geq 1$ and for all $u \in B_m(X)$ and $v \in B_n(X)$ such that $uv \in B_{m+n}(X)$, we have
\begin{equation*}
\mu(uv) \leq K \mu(u) \mu(v);
\end{equation*}
\item for all $m,n \geq 1$ and for all $u \in B_m(X)$ and $v \in B_n(X)$ such that $uv \in B_{m+n}(X)$, we have
\begin{equation*}
\mu \bigl( \sigma^{-m}[v] \mid [u] \bigr) \leq K \mu \bigl( [v] \bigr);
\end{equation*}
\item for $g \geq g_0$, for all $m,n \geq 1$ and for all $u \in B_m(X)$ and $v \in B_n(X)$, we have
\begin{equation*}
\mu\Bigl( [u] \bigcap \sigma^{-g+m}[v]\Bigr) \geq K^{-1} \mu( [u] ) \mu( [v] ).
\end{equation*}
\end{itemize}
\end{lemma}

\subsection{Basics of random SFTs} \label{Sect:IntroRSFTs}

Let $X$ be a non-trivial mixing SFT. Fix $\alpha \in (0,1)$. Recall that $\mathcal{F}_n$ denotes the random subset of $B_n(X)$ formed by including each word with probability $1-\alpha$, independently of all other words, and $Y_n = Y(\mathcal{F}_n)$ is the random SFT formed by forbidding the words $\mathcal{F}_n$ from $X$. Here we establish some notation and basic facts for random SFTs.

Let $u \in B_k(X)$ for some $k \geq n$. We let $W_n(u)$ denote the set of all words of length $n$ that appear in $u$:
\begin{equation*}
W_n(u) = \bigl\{ v \in B_n(X) : \exists j \in \{1,\dots,k-n+1\}, \, u_{j}^{j+n-1} = v \bigr\}.
\end{equation*}
Then let $\xi_u$ denote the indicator function of the event that $u$ contains no words from $\mathcal{F}_n$, \textit{i.e.},
\begin{equation*}
\xi_u = \left\{ \begin{array}{ll}
                          1, & \text{ if } W_n(u) \cap \mathcal{F}_n = \varnothing \\
                          0, & \text{ otherwise}.
                     \end{array} \right.
\end{equation*} 
Since each word in $W_n(u)$ is excluded from $\mathcal{F}_n$ with probability $\alpha$, independently of all other words, we have that
\begin{equation} \label{Eqn:BasicExp}
\mathbb{E}\bigl[ \xi_u \bigr] = \mathbb{P} \bigl( W_n(u) \cap \mathcal{F}_n = \varnothing \bigr) = \alpha^{|W_n(u)|}.
\end{equation}
Furthermore, for $u,v \in B_k(X)$, the covariance of $\xi_u$ and $\xi_v$ is given by
\begin{align} \label{Eqn:BasicVar}
\begin{split}
\mathbb{E}\biggl[ \bigl(\xi_u - \mathbb{E}[\xi_u] \bigr) \bigl( \xi_v - \mathbb{E}[\xi_v] \bigr) \biggr] & = \mathbb{E} \bigl[ \xi_u \xi_v \bigr] - \mathbb{E}\bigl[ \xi_u \bigr] \mathbb{E} \bigl[ \xi_v \bigr]  \\
& = \alpha^{|W_n(u) \cup W_n(v)|} - \alpha^{|W_n(u)| + |W_n(v)|} \\
& =   \alpha^{|W_n(u) \cup W_n(v)|} \bigl( 1- \alpha^{|W_n(u) \cap W_n(v)|} \bigr) .
\end{split}
\end{align}

\begin{remark} \label{Rmk:Measure}
Let $X$ be a non-trivial mixing SFT, and let $\mu \in M(X,\sigma)$. Then the expected value of the $\mu$-measure of the hole $H_n$ is $1-\alpha$, since
\begin{align*}
\mathbb{E} \bigl[ \mu(H_n) \bigr] & = \mathbb{E}\Biggl[ \sum_{u \in B_n(X)} \mu(u) (1- \xi_u) \Biggr] \\
& = \sum_{u \in B_n(X)} \mu(u) (1- \mathbb{E}\bigl[ \xi_u \bigr] ) = (1- \alpha) \, \mu(B_n(X)) = 1- \alpha.
\end{align*}
\end{remark}

\subsection{Repeats and repeat covers}
We use interval notation to denote intervals in $\mathbb{Z}$. For example, $[1,3] = \{1,2,3\}$ and $[0,5) = \{0,1,2,3,4\}$. Furthermore, for a set $F \subset \mathbb{Z}$ and $t \in \mathbb{Z}$, we let $t+F = \{ t+s : s \in F\}$. For $n$ in $\mathbb{N}$ and $F \subset \mathbb{Z}$, we let $\mathcal{C}_n(F)$ denote the set of intervals of length $n$ contained in $F$:
\begin{equation*}
\mathcal{C}_n(F) = \bigl\{ t + [0,n) : t+[0,n)\subset F \bigr\}.
\end{equation*}
We will also consider sets of pairs of intervals; that is, we consider sets $\mathcal{R} \subset \mathcal{C}_n(F) \times \mathcal{C}_n(F)$. For such $\mathcal{R}$, we let $|\mathcal{R}|$ denote the number of pairs in $\mathcal{R}$, and we let
\begin{equation*}
A(\mathcal{R}) = \bigcup_{(I_1,I_2) \in \mathcal{R}} I_2.
\end{equation*}
Now we define \textit{repeats} and \textit{repeat covers}, which were used implicitly in \cite{McGoff2012} and then defined explicitly in \cite{McGoffPavlov2016}.
\begin{definition}
Let $\mathcal{A}$ be a finite set. Let $F \subset \mathbb{Z}$, and let $u \in \mathcal{A}^F$. A pair $(I_1,I_2)$ in $\mathcal{C}_n(F) \times \mathcal{C}_n(F)$ is an $n$-\textit{repeat} (or just a \textit{repeat})  for $u$ if $u_{I_1} = u_{I_2}$ and $I_1$ is the lexicographically minimal occurrence of the word $u_{I_1}$ in $u$. In that case, the word $u_{I_1}$ is called a repeated word for $u$. Furthermore, a set $\mathcal{R} \subset \mathcal{C}_n(F) \times \mathcal{C}_n(F)$ is a \textit{repeat cover} for $u$ if
\begin{enumerate}
\item each pair $(I_1,I_2) \in \mathcal{R}$ is a repeat for $u$, and
\item for each repeat $(I_1,I_2)$ for $u$, we have $I_2 \subset A(\mathcal{R})$.
\end{enumerate}
\end{definition}

Note that every pattern $u \in \mathcal{A}^F$ has a repeat cover, which contains all repeats for $u$. However, in many cases, we seek to find more efficient repeat covers, by which we mean repeat covers $\mathcal{R}$ such that $|\mathcal{R}|$ is small enough for our purposes. In this paper, we only require the bound supplied by the following lemma, which is  a slightly weaker version of Lemma 3.8 in \cite{McGoffPavlov2016}.
\begin{lemma} \label{Lemma:EfficientRepeatCovers}
Let $F \subset \mathbb{Z}$ be a finite union of intervals of length $n$, and suppose $u \in \mathcal{A}^F$. Then there exists an $n$-repeat cover $\mathcal{R}$ for $u$ such that
\begin{equation*}
\bigl|  \mathcal{R} \bigr| \leq 4|F|/n.
\end{equation*}
\end{lemma}

Additionally, our proofs make use of the following estimate relating the number of unique words of length $n$ in $u \in B_k(X)$ and the cardinality of the repeat area for $u$.\begin{lemma} \label{Lemma:RepeatArea}
Let $F \subset \mathbb{Z}$ be a finite union of intervals of length $n$, and let $a = |\{ s : s+ [0,n) \subset F\}|$. Suppose that $u \in \mathcal{A}^F$ satisfies $|W_n(u)| = j < a$. Then for any repeat cover $\mathcal{R}$ for $u$, 
\begin{equation*}
|A(\mathcal{R})| \geq  a+n - j - 1.
\end{equation*}
\end{lemma}
\begin{proof}
Let $r = a - j$, which is the number of repeats for $u$. The lexicographically minimal repeat for $u$ contributes $n$ elements to $A(\mathcal{R})$, and each of the other $r-1$ repeats must contribute at least one element. Altogether, we must have $|A(\mathcal{R})| \geq n + r-1 = n+a-j-1$.
\end{proof}

In many of the proofs in Section \ref{Sect:RepeatProofs}, we decompose words into alternating blocks of repeated regions (\textit{i.e.}, regions contained in $A(\mathcal{R})$ for some repeat cover $\mathcal{R}$) and non-repeated regions. The following definition standardizes some notation that is useful for such decompositions. We endow $\mathbb{Z}$ with the standard graph structure, in which two nodes $x,y \in \mathbb{Z}$ are connected by an edge whenever $|x-y| = 1$. We then endow all subsets of $\mathbb{Z}$ with the induced subgraph structure, and references to connected components refer to this subgraph structure. Furthermore, we give $\mathbb{Z}$ the standard ordering, and if $I$ and $J$ are disjoint subsets of $\mathbb{Z}$, then we let $I < J$ whenever $x < y$ for all $x \in I$ and $y \in J$.

\begin{definition} \label{Def:Blocks}
Let $A \subset [0,k)$ be a union of intervals of length $n$ such that $0 \notin A$. Then the \textit{interval decomposition} of $[0,k)$ induced by $A$ consists of $\bigl( (I_m)_{m=1}^{N+1}, (J_{m})_{m=1}^N \bigr)$, where
\begin{itemize}
\item each $J_m$ is a non-empty maximal connected component of $A$, and $\bigcup_m J_m = A$;
\item each $I_m$ is a maximal connected component of $[0,k) \setminus A$, and $\bigcup_m I_m = [0,k) \setminus A$;
\item only $I_{N+1}$ may be empty;
\item for each $m =1, \dots, N$, we have $I_m < J_m < I_{m+1}$.
\end{itemize}

Now suppose $b \in \mathcal{A}^k$ and $\mathcal{R} \subset \mathcal{C}_n([0,k)) \times \mathcal{C}_n([0,k))$. Let $A = A(\mathcal{R})$, and let $\bigl( (I_m)_{m=1}^{N+1}, (J_{m})_{m=1}^N \bigr)$ be the interval decomposition of $[0,k)$ induced by $A$. For each $m$, we let $u_m = b|_{I_m}$ and $v_m = b|_{J_m}$.  We refer to $\bigl( (u_m)_{m=1}^{N+1}, (v_m)_{m=1}^N \bigr)$ as the \textit{block decomposition} of $b$. If $\mathcal{R}$ is a repeat cover for $b$, then we refer to $\bigl( (u_m)_{m=1}^{N+1}, (v_m)_{m=1}^N \bigr)$ as the \textit{repeat block decomposition} of $b$. Note that $N \leq |\mathcal{R}|$.
\end{definition}

When $\mu$ is a Gibbs measure associated to a H\"{o}lder continuous potential, the following lemma, which is used several times in Section \ref{Sect:RepeatProofs}, gives an estimate of the $\mu$-measure of any word $b$ in terms of a block decomposition.
\begin{lemma} \label{Lemma:Birthday}
Let $X$ be a non-trivial mixing SFT with H\"{o}lder continuous potential $f : X \to \R$ and associated Gibbs measure $\mu$. Let $K>0$ satisfy the conclusions of Lemma \ref{Lemma:Oscar}.
Let $b \in B_k(X)$, and suppose that $\bigl( (u_m)_{m=1}^{N+1}, (v_m)_{m=1}^N)$ is a block decomposition of $b$. Then
\begin{align*}
\mu( b ) \leq K^{2N} \prod_{m=1}^{N+1} \mu ( u_m ) \prod_{m=1}^N \mu( v_m).
\end{align*}
\end{lemma}
\begin{proof}
Let $\bigl( (I_m)_{m=1}^{N+1}, (J_{m})_{m=1}^N \bigr)$ be a block decomposition of $b$.
Let $s_m$ be the minimal element of the corresponding interval $I_m$, and let $t_m$ be the minimal element of the interval $J_m$. 
In order to avoid confusion, in this proof we use proper cylinder set notation: for $u \in B_m(X)$, we let $[u]$ denote the set of points $x$ in $X$ such that $x_0 \dots x_{m-1} = u$. 
Using conditional probabilities, we have
\begin{align*}
\mu([b]) = \mu([u_1]) \prod_{m=1}^{N} \mu\bigl( \sigma^{-t_m}[v_m] \mid [u_1 \dots u_m] \bigr) \prod_{m=1}^{N} \mu \bigl( \sigma^{-s_{m+1}}[u_{m+1}] \mid [u_1 \dots v_m] \bigr) .
\end{align*}
Then by our choice of $K$, we have
\begin{equation*}
\mu( [b] ) \leq K^{2N} \prod_{m=1}^{N+1} \mu ( [u_m] ) \prod_{m=1}^N \mu( [v_m]),
\end{equation*}
as desired.
\end{proof}

\section{Pressure of random SFTs}  \label{Sect:MainProof}

In this section we give a proof of Theorem \ref{Thm:Pressure}. The broad outline of the proof involves finding upper and lower bounds on the pressure in terms of some other random variables, followed by a second moment argument showing that these auxiliary random variables each converge in probability to $P_X(f)+\log \alpha$. For the sake of exposition, we present the argument here and defer the proofs of the many technical lemmas to later sections of the paper. We hope that this presentation helps clarify the main argument and also motivate the technical lemmas. Note that at the beginning of this proof we define some notation and choose some parameters, including the random variables $\phi_{n,k}$ and $\psi_{n,k}$, and we make frequent reference to both the notation and the parameters throughout Sections \ref{Sect:RepeatProofs} - \ref{Sect:UBandLB} in the technical lemmas.

\vspace{2mm}

\begin{PfOfThmPressure}

Let $X$ be a non-trivial mixing SFT. Let $f : X \to \R$ be a H\"{o}lder continuous potential with associated Gibbs measure $\mu$. Choose $\gamma_0 = \gamma_0(X,f)$ as in Definition \ref{Def:Gamma}, and note that $\gamma_0 < 1$. Let $\alpha \in (\gamma_0, 1]$, and let $\epsilon >0$. Furthermore, fix $K$ and $g_0$ as in Lemma \ref{Lemma:Oscar}.

We begin by selecting a variety of parameters for our proof. Since $\alpha > \gamma_0$, there exists $\gamma$ in the interval $(\gamma_0, \alpha)$. According to the definition of $\gamma_0$, since $\gamma > \gamma_0$, there exists $n_0$ such that for all $m \geq n_0$, for all $u \in B_m(X)$, we have $\mu(u) \leq \gamma^m$. We assume throughout that $n \geq n_0$. Choose $\delta >0$ such that $\delta < \log(\alpha \gamma^{-1})/4$. Fix a sequence $k = k(n)$ such that $n/k \to 0$ and $k = o(n^2/\log n)$. (For example, one may choose $k = [n^{1+\nu}]$ for any $0 < \nu < 1$.) Now let $\ell = k-n+1$, which is the number of positions $s$ in $[0,k)$ such that $s+[0,n) \subset [0,k)$. 

Having made these parameter choices, we now proceed to define our upper and lower bounds on $P_{Y_n}(f)$. First, for all $m \geq n$, for all $u \in B_m(X)$, recall from Section \ref{Sect:IntroRSFTs} that $\xi_u$ is the random variable that is one if $u$ is allowed (\textit{i.e.} $W_n(u) \cap \mathcal{F}_n = \varnothing$) and zero otherwise. Then define
\begin{equation*}
\phi_{n,k} = \sum_{u \in B_k(X)} e^{S_k f(u)} \xi_u.
\end{equation*}
By Lemma \ref{Lemma:UB}, $\phi_{n,k}$ may be used to provide an upper bound on $P_{Y_n}(f)$:
\begin{equation} \label{Eqn:Florence}
P_{Y_n}(f) \leq \frac{1}{k} \log \phi_{n,k}.
\end{equation}

Now we turn towards the lower bound on $P_{Y_n}(f)$. Recall that we have already defined $\delta>0$ above. Consider the set of words of length $n$ that are entropy-typical for $\mu$ with tolerance $\delta$:
\begin{equation*}
E_n = \biggl\{ u \in B_n(X) : \biggl| - \frac{1}{n} \log \mu(u) - h(\mu) \biggr| < \delta \biggr\}.
\end{equation*}
Then let $G_{n,k}$ be the set of words of length $k$ that begin and end with the same word of length $n$ from $E_n$:
\begin{equation*}
G_{n,k} = \Bigl\{ u \in B_k(X) : u_1^n = u_{\ell}^k \text{ and } u_1^n \in E_n \Bigr\}.
\end{equation*}
Next we define the random variable
\begin{equation*}
\psi_{n,k} = \frac{1}{|E_n|} \sum_{u \in G_{n,k}} e^{S_kf(u)} \xi_u.
\end{equation*}
By Lemma \ref{Lemma:LB}, we see that $\psi_{n,k}$ may be used to bound $P_{Y_n}(f)$ from below: for all large enough $n$,
\begin{equation} \label{Eqn:Nightengale}
P_{Y_n}(f) \geq \frac{1}{k} \log \psi_{n,k} - \epsilon/2.
\end{equation}

By Lemmas \ref{Lemma:Skate}, \ref{Lemma:Board}, \ref{Lemma:Roller}, and \ref{Lemma:Blade}, we have the following asymptotic results on the expectation and variance of $\phi_{n,k}$ and $\psi_{n,k}$:
\renewcommand{\theenumi}{\Roman{enumi}}
\begin{enumerate}
\item $\lim_n \frac{1}{k} \log \mathbb{E}\bigl[ \phi_{n,k} \bigr] = P_X(f) + \log( \alpha)$; \label{Item:1}
\item $\lim_n \frac{1}{k} \log \mathbb{E}\bigl[ \psi_{n,k} \bigr] = P_X(f) + \log( \alpha)$;  \label{Item:2}
\item there exists $\rho_1 > 0$ such that for all large enough $n$, \label{Item:3}
\begin{equation*}
\frac{ \Var \bigl[ \phi_{n,k} \bigr] }{ \mathbb{E} \bigl[ \phi_{n,k} \bigr]^2 } \leq e^{- \rho_1 n};
\end{equation*}
\item there exists $\rho_2 > 0$ such that for all large enough $n$, \label{Item:4}
\begin{equation*}
\frac{ \Var \bigl[ \psi_{n,k} \bigr] }{ \mathbb{E} \bigl[ \psi_{n,k} \bigr]^2 } \leq e^{- \rho_2 n}.
\end{equation*}
\end{enumerate}
The first two properties indicate that we expect $\phi_{n,k}$ and $\psi_{n,k}$ to be on the correct exponential order of magnitude, while the third and fourth properties show that these random variables are well concentrated around their expected values. Combining these properties with Chebyshev's inequality, we are able to finish the proof as follows.

By the monotonicity of $\mathbb{P}$ under inclusion, the union bound, and the displays (\ref{Eqn:Florence}) and (\ref{Eqn:Nightengale}), for all large enough $n$, we have
\begin{align}  \label{Eqn:France}
\begin{split}
\mathbb{P} & \biggl( | P_{Y_n}(f) - (P_X(f)+\log \alpha) | \geq \epsilon \biggr) \\
 & \leq \mathbb{P} \biggl( P_{Y_n}(f) \geq P_X(f) + \log \alpha + \epsilon \biggr) + \mathbb{P} \biggl( P_{Y_n}(f) \leq P_X(f) + \log \alpha - \epsilon \biggr) \\
 & \leq \mathbb{P} \biggl( \frac{1}{k} \log \phi_{n,k} \geq P_X(f) + \log \alpha + \epsilon \biggr) + \mathbb{P} \biggl( \frac{1}{k} \log \psi_{n,k} - \epsilon/2 \leq P_X(f) + \log \alpha - \epsilon \biggr) \\
 & = \mathbb{P} \biggl(  \phi_{n,k} \geq e^{k (P_X(f) + \log \alpha + \epsilon)} \biggr) + \mathbb{P} \biggl(  \psi_{n,k} \leq e^{k(P_X(f) + \log \alpha - \epsilon/2)} \biggr).
\end{split}
\end{align}
We proceed to bound the two terms on the right-hand side separately. For the first, Chebyshev gives 
\begin{align*}
\mathbb{P} & \biggl(  \phi_{n,k} \geq e^{k (P_X(f) + \log \alpha + \epsilon)} \biggr) \\
& = \mathbb{P} \biggl(  \phi_{n,k} - \mathbb{E}[\phi_{n,k}] \geq e^{k (P_X(f) + \log \alpha + \epsilon)} - \mathbb{E}[\phi_{n,k}] \biggr) \\
& = \mathbb{P} \biggl(  \phi_{n,k} - \mathbb{E}[\phi_{n,k}] \geq \Var[\phi_{n,k}]^{1/2} \frac{\mathbb{E}[\phi_{n,k}]}{\Var[\phi_{n,k}]^{1/2}} \Bigl(e^{k (P_X(f) + \log \alpha + \epsilon)}/\mathbb{E}[\phi_{n,k}] - 1 \Bigr) \biggr) \\
& \leq  \frac{\Var[\phi_{n,k}]}{\mathbb{E}[\phi_{n,k}]^2} \Bigl(e^{k (P_X(f) + \log \alpha + \epsilon)}/\mathbb{E}[\phi_{n,k}] - 1 \Bigr)^{-2}.
\end{align*}
Then by properties (\ref{Item:1}) and (\ref{Item:3}), there exists $\rho_3 >0$ such that for all large $n$,
\begin{equation} \label{Eqn:Italy}
\mathbb{P} \biggl(  \phi_{n,k} \geq e^{k (P_X(f) + \log \alpha + \epsilon)} \biggr)  < e^{-\rho_3 n}.
\end{equation}
Similarly, for the second term in the last line of (\ref{Eqn:France}), Chebyshev gives
\begin{align*}
\mathbb{P} & \biggl(  \psi_{n,k} \leq e^{k (P_X(f) + \log \alpha -\epsilon/2)} \biggr) \\
& = \mathbb{P} \biggl(  \psi_{n,k} - \mathbb{E}[\psi_{n,k}] \leq e^{k (P_X(f) + \log \alpha - \epsilon/2)} - \mathbb{E}[\psi_{n,k}] \biggr) \\
& = \mathbb{P} \biggl(  \psi_{n,k} - \mathbb{E}[\psi_{n,k}] \leq \Var[\psi_{n,k}]^{1/2} \frac{\mathbb{E}[\psi_{n,k}]}{\Var[\psi_{n,k}]^{1/2}} \Bigl(e^{k (P_X(f) + \log \alpha - \epsilon/2)}/\mathbb{E}[\psi_{n,k}] - 1 \Bigr) \biggr) \\
& \leq  \frac{\Var[\psi_{n,k}]}{\mathbb{E}[\psi_{n,k}]^2} \Bigl(e^{k (P_X(f) + \log \alpha - \epsilon/2)}/\mathbb{E}[\psi_{n,k}] - 1 \Bigr)^{-2}.
\end{align*}
Then by properties (\ref{Item:2}) and (\ref{Item:4}), there exists $\rho_4 >0$ such that for all large $n$, 
\begin{equation} \label{Eqn:Spain}
\mathbb{P}  \biggl(  \psi_{n,k} \leq e^{k (P_X(f) + \log \alpha -\epsilon/2)} \biggr) < e^{-\rho_4 n}.
\end{equation}
Combining the inequalities in (\ref{Eqn:France}), (\ref{Eqn:Italy}), and (\ref{Eqn:Spain}), we obtain the desired result.

\end{PfOfThmPressure}

\section{Repeat probabilities} \label{Sect:RepeatProofs}

%

In this section we bound the $\mu$-measure of sets of words that have repeated subwords. By the well-known result of Ornstein and Weiss \cite{OW}, the first return time of a $\mu$-typical point $x$ to its initial block of length $n$ is approximately $e^{h(\mu)n}$. Then for $\mu$-typical words of polynomial length in $n$, one would expect to find no repeated words of length $n$ at all. However, to control the expectation and variance of $\phi_{n,k}$ and $\psi_{n,k}$, it is important to give more precise estimates on just how unlikely it is that a word of length $k$ will have exactly $j$ distinct subwords of length $n$, for each $1 \leq j \leq k-n+1$.

Throughout this section we use the same environment (notation, parameters, and assumptions) laid out at the beginning of the proof of Theorem \ref{Thm:Pressure}. The results of this section are used in the following section to establish properties (\ref{Item:1}) - (\ref{Item:4}) from the proof of Theorem \ref{Thm:Pressure}. 
%

We begin by considering some sets of words that have exactly $j$ distinct subwords of length $n$.
For $1 \leq j \leq \ell$, we define the following sets:
\begin{align*}
B_{n,k}^j & = \bigl\{ u \in B_k(X) : |W_n(u)| = j \bigr\} \\
G_{n,k}^j & = \bigl\{ u \in G_{n,k} : |W_n(u)| = j \bigr\}.
\end{align*}
Furthermore, for $1 \leq j \leq 2 \ell$, we let
\begin{align*}
D_{n,k}^j & = \bigl\{ (u,v) \in B_k(X) \times B_k(X) : W_n(u) \cap W_n(v) \neq \varnothing, \, |W_n(u) \cup W_n(v) | = j \bigl\} \\
Q_{n,k} & =  \bigl\{ (u,v) \in G_{n,k} \times G_{n,k} : W_n(u) \cap W_n(v) \neq \varnothing \bigl\} \\
Q_{n,k}^j & =  \bigl\{ (u,v) \in Q_{n,k} :  |W_n(u) \cup W_n(v) | = j \bigl\}.  
\end{align*}

In Lemmas \ref{Lemma:NPS} and \ref{Lemma:NewAsia} we find bounds on the $\mu$-measure of the sets $B_{n,k}^j$ and $G_{n,k}$. In subsequent lemmas (Lemmas \ref{Lemma:Super} - \ref{Lemma:Peace}), we also find bounds on the $\mu \otimes \mu$-measure of $D_{n,k}^j$, $Q_{n,k}$, and $Q_{n,k}^j$. These estimates are used in the following section to bound the expectation and variance of $\phi_{n,k}$ and $\psi_{n,k}$.

\begin{lemma} \label{Lemma:NPS}
There exists a polynomial $p_1(x)$ such that for all large enough $n$, for each  $1 \leq j \leq \ell$, we have
\begin{equation*}
\mu \bigl( B_{n,k}^j \bigr) \leq p_1(n)^{k/n} \gamma^{k-j}.
\end{equation*}
\end{lemma}
\begin{proof}
Consider $n \geq n_0$ and  and $1 \leq j \leq \ell$.
We define a map $\varphi : B_{n,k}^j \to \{ (\mathcal{R},w) : \mathcal{R} \subset \mathcal{C}_n([0,k)) \times \mathcal{C}_n([0,k)), \, w \in \mathcal{A}^{[0,k) \setminus A(\mathcal{R})} \}$ as follows. 
Let $b$ be in  $B_{n,k}^j$. Let $\mathcal{R}$ be a repeat cover of $b$ such that $|\mathcal{R}| \leq 4k/n$, which exists by Lemma \ref{Lemma:EfficientRepeatCovers}. Let $\bigl( (u_m)_{m=1}^{N+1}, (v_m)_{m=1}^N \bigr)$ be a repeat block decomposition of $b$ (as in Definition \ref{Def:Blocks}). Then set $\varphi(b) = (\mathcal{R}, (v_m)_{m=1}^N)$.
Furthermore, note that by Lemma \ref{Lemma:Birthday},
\begin{equation*}
\mu(b) \leq K^{2N} \prod_{m=1}^N \mu(v_m) \prod_{m=1}^{N+1} \mu( u_m).
\end{equation*}
Since each block $v_m$ has length at least $n$ (it contains at least one repeated word of length $n$ from $b$) and $n \geq n_0$, we have that $\mu(v_m) \leq \gamma^{|v_m|}$.  Then
\begin{align*} 
\mu(b)  \leq K^{2N} \gamma^{\sum_m |v_m|} \prod_{m=1}^{N+1} \mu( u_m) = K^{2N} \gamma^{|A(\mathcal{R})|} \prod_{m=1}^{N+1} \mu( u_m)
\end{align*}
Using that $N \leq |\mathcal{R}| \leq 4k/n$ and $|A(\mathcal{R})| \geq k-j$ (by Lemma \ref{Lemma:RepeatArea}), we see that
\begin{align} \label{Eqn:Firenze}
\begin{split}
\mu(b) & \leq \bigl(K^8 \bigr)^{k/n} \gamma^{ |A(\mathcal{R})| }  \prod_{m=1}^{N+1} \mu( u_m) \\
 & \leq \bigl(K^8 \bigr)^{k/n} \gamma^{ k-j }  \prod_{m=1}^{N+1} \mu( u_m).
 \end{split}
\end{align}

Now define the projection map $\pi : \varphi(B_{n,k}^j) \to \{ \mathcal{R} : \mathcal{R} \subset \mathcal{C}_n(k) \times \mathcal{C}_n(k) \}$, given by $\pi ( (\mathcal{R} , w) ) = \mathcal{R}$. Let $S = \pi \circ \varphi(B_{n,k}^j)$. Note that $|\mathcal{C}_n([0,k))| \leq k$, and therefore $|\mathcal{C}_n([0,k)) \times \mathcal{C}_n([0,k))| \leq k^2$. Furthermore, since each $\mathcal{R}$ in $\pi \circ \varphi(B_{n,k}^j)$ satisfies $|\mathcal{R}| \leq 4 k/n$, we have that $|S| \leq |\mathcal{C}_n([0,k)) \times \mathcal{C}_n([0,k))|^{4k/n} \leq \bigl(k^2\bigr)^{4k/n} = \bigl(k^8 \bigr)^{k/n}$.

Having established these bounds, we may now estimate the $\mu$-measure of $B_{n,k}^j$ as follows. By rearranging the sum, we have
\begin{align*}
\mu(B_{n,k}^j) & = \sum_{b \in B_{n,k}^j} \mu(b) \\
& =  \sum_{\mathcal{R} \in S} \; \sum_{ (\mathcal{R},(u_m)) \in \pi^{-1}(\mathcal{R})} \; \sum_{b \in \varphi^{-1}(\mathcal{R},(u_m))} \mu(b).
\end{align*}
Then by (\ref{Eqn:Firenze}), we get
\begin{align*}
\mu(B_{n,k}^j) & \leq \sum_{\mathcal{R} \in S} \; \sum_{ (\mathcal{R},(u_m)) \in \pi^{-1}(\mathcal{R})} \; \sum_{b \in \varphi^{-1}(\mathcal{R},(u_m))} \bigl(K^8 \bigr)^{k/n} \gamma^{ k-j }  \prod_{m=1}^{N+1} \mu( u_m) \\
& = \bigl(K^8 \bigr)^{k/n} \gamma^{ k-j } \sum_{\mathcal{R} \in S} \;  \sum_{ (\mathcal{R},(u_m)) \in \pi^{-1}(\mathcal{R})} \;  \prod_{m=1}^{N+1} \mu( u_m).
\end{align*}
Since $\mu$ is a probability measure, the sum of $\mu(u_m)$ over any set of words $u_m$ of the same length is less than or equal to one. Then
\begin{align*}
\mu(B_{n,k}^j) & \leq \bigl(K^8 \bigr)^{k/n} \gamma^{ k-j } \sum_{\mathcal{R} \in S} \;  \sum_{ (\mathcal{R},(u_m)) \in \pi^{-1}(\mathcal{R})} \;  \prod_{m=1}^{N+1} \mu( u_m) \\
& \leq \bigl(K^8 \bigr)^{k/n} \gamma^{ k-j }  |S| \\
& \leq \bigl(K^8 k^8 \bigr)^{k/n} \gamma^{k-j},
\end{align*}
where we have used that $|S| \leq \bigl(k^8 \bigr)^{k/n}$ (established in the previous paragraph).

Recall that $k = o(n^2/\log(n))$. Therefore for large enough $n$, we have $k \leq n^2$. 
Let $p_1(x) = K^8 x^{16}$. 
Then for large enough $n$, for all $1 \leq j \leq \ell$, the previous display yields that $\mu(B_{n,k}^j) \leq p_1(n)^{k/n} \gamma^{k-j}$, as desired.
\end{proof}

The following lemma gives both upper and lower bounds on the $\mu$-measure of $G_{n,k}$.
\begin{lemma} \label{Lemma:NewAsia}
There exists $\rho_0 >0$ such that for all large enough $n$, 
\begin{equation*}
K^{-1} e^{- (h(\mu) + \delta) n} (1 - e^{- \rho_0 n})
 \leq \mu( G_{n,k} ) \leq K e^{-(h(\mu)- \delta) n}.
  \end{equation*}
\end{lemma}
\begin{proof}
For $u \in E_n$, let $G_{n,k}(u) = \{ v \in G_{n,k} : v_1^n = u\}$.
Then by our choice of $K$, for large enough $n$, we have
\begin{equation*}
\mu( G_{n,k}(u) ) = \sum_{v \in G_{n,k}(u)} \mu(v) \geq K^{-1} \mu(u)^2.
\end{equation*}
Also, note that $G_{n,k} = \sqcup_{u \in E_n} G_{n,k}(u)$. Then
\begin{equation*}
\mu(G_{n,k}) = \sum_{u \in E_n} \mu(G_{n,k}(u)) \geq K^{-1} \sum_{u \in E_n} \mu(u)^2 \geq K^{-1} e^{-(h(\mu) + \delta) n} \mu(E_n),
\end{equation*}
where the last inequality results from the fact that $\min_{u \in E_n} \mu(u) \geq e^{-(h(\mu)+\delta)n}$.
Additionally, using the Gibbs property (\ref{Eqn:GibbsProp}) and the large deviations results for Gibbs measures \cite{Young1990}, one may check that there exists $\rho_0 >0$ such that $\mu(E_n) \geq 1- e^{ - \rho_0 n}$ for all large enough $n$. Combining this fact with the above inequalities yields the desired lower bound.

For the upper bound, for all large enough $n$ and for each $u \in E_n$, we have that
\begin{equation*}
\mu( G_{n,k}(u) ) = \sum_{v \in G_{n,k}(u)} \mu(v) \leq K \mu(u)^2.
\end{equation*}
Then
\begin{equation*}
\mu(G_{n,k}) = \sum_{u \in E_n} \mu(G_{n,k}(u)) \leq K \sum_{u \in E_n} \mu(u)^2 \leq K e^{-(h(\mu) - \delta) n},
\end{equation*}
where we have used that $\max_{u \in E_n} \mu(u) \leq e^{-(h(\mu)-\delta)n}$. 
\end{proof}

In the following three lemmas, we estimate the $\mu \otimes \mu$ measure of sets of pairs of words with various repeat properties. The outline of these proofs is similar to the outline of the proof of Lemma \ref{Lemma:NPS}, but each proof requires some arguments that are specific to the particular repeat structure of interest.

\begin{lemma} \label{Lemma:Super}
There exists a polynomial $p_2(x)$ such that for all large enough $n$, for all $1 \leq j \leq 2\ell-1$,
\begin{equation*}
\mu \otimes \mu(D_{n,k}^j) \leq p_2(n)^{k/n} \gamma^{2 \ell - j + n}.
\end{equation*}
\end{lemma}
\begin{proof}
Consider $n \geq n_0$, and let $1 \leq j \leq 2\ell-1$. First define the set $F = [0,k) \sqcup [k+1,2k]$. Then define a map $\varphi : D_{n,k}^j \to \{ (\mathcal{R}, w) : \mathcal{R} \subset \mathcal{C}_n(F) \times \mathcal{C}_n(F), w \in \mathcal{A}^{F \setminus A(\mathcal{R})} \}$ as follows.
Let $(a,b) \in D_{n,k}^j$. We use the notation $a \sqcup b$ to denote the element in $\mathcal{A}^{F}$ such that $(a \sqcup b)|_{[0,k)} = a$ and $(a \sqcup b)|_{[k+1,2k]} = b$. Let $\mathcal{R}$ be a repeat cover of $a \sqcup b$ such that $|\mathcal{R}| \leq 4|F|/n = 8k/n$, which exists by Lemma \ref{Lemma:EfficientRepeatCovers}. Then let $\bigl( (u_m)_{m=1}^{N_1+1}, (v_m)_{m=1}^{N_1} \bigr)$ be the repeat block decomposition of $a$ induced by the set $A(\mathcal{R}) \cap [0,k)$, and let $\bigl( (y_m)_{m=1}^{N_2+1}, (z_m)_{m=1}^{N_2} \bigr)$ be the repeat block decomposition of $b$ induced by the set $A(\mathcal{R}) \cap [k+1,2k]$.
Finally, we define $\varphi(a,b) = (\mathcal{R}, (u_m)_{m=1}^{N_1+1}, (y_m)_{m=1}^{N_2+1} )$.

By Lemma \ref{Lemma:Birthday}, note that
\begin{equation*}
\mu(a) \leq K^{2N_1} \prod_{m=1}^{N_1} \mu(v_m) \prod_{m=1}^{N_1+1} \mu( u_m),
\end{equation*}
and
\begin{equation*}
\mu(b) \leq K^{2N_2} \prod_{m=1}^{N_2} \mu(z_m) \prod_{m=1}^{N_2+1} \mu( y_m).
\end{equation*}
Furthermore, since each of the blocks $v_m$ and $z_m$ has length at least $n$, for all large enough $n$, we have that $\mu(v_m) \leq \gamma^{|v_m|}$ and $\mu(z_m) \leq \gamma^{|z_m|}$. Therefore for all large enough $n$, we have
\begin{equation*}
\mu(a) \leq K^{2N_1} \gamma^{\sum_m |v_m|} \prod_{m=1}^{N_1+1} \mu( u_m), 
\end{equation*}
and
\begin{equation*}
\mu(a) \leq K^{2N_2} \gamma^{\sum_m |z_m|} \prod_{m=1}^{N_2+1} \mu( y_m). 
\end{equation*}
Using that $N_1+N_2 \leq |\mathcal{R}| \leq 8k/n$ and $\sum_m |v_m| + \sum_m |z_m| = |A(\mathcal{R})| \geq 2k-j -n$ (by Lemma \ref{Lemma:RepeatArea}), we obtain
\begin{equation} \label{Eqn:BlueFish}
\mu(a) \mu(b) \leq \bigl(K^{16} \bigr)^{k/n} \gamma^{2k-j-n} \prod_{m=1}^{N_1+1} \mu( u_m) \prod_{m=1}^{N_2+1} \mu( y_m).
\end{equation}

Now define the projection $\pi : \varphi(D_{n,k}^j) \to \{ \mathcal{R} : \mathcal{R} \subset \mathcal{C}_n(F) \times \mathcal{C}_n(F) \}$, given by $\pi( \mathcal{R}, (u_m), (y_m) ) = \mathcal{R}$. Let $S = \pi \circ \varphi(D_{n,k}^j)$. Note that $|\mathcal{C}_n(F)| \leq 2k$, and so $|\mathcal{C}_n(F) \times \mathcal{C}_n(F)| \leq (2k)^2$. Since each $\mathcal{R} \in \pi \circ \varphi(D_{n,k}^j)$ satisfies $|\mathcal{R}| \leq 8k/n$, we then have that $|S| = |\pi \circ \varphi(D_{n,k}^j)| \leq |\mathcal{C}_n(F) \times \mathcal{C}_n(F)|^{8k/n} \leq (2k)^{16k/n}$. 

Let us now estimate $\mu \otimes \mu( D_{n,k}^j)$. By rearranging the sum, we get
\begin{align*}
\mu \otimes \mu( D_{n,k}^j ) & = \sum_{(a,b) \in D_{n,k}^j} \mu(a) \mu(b) \\
& = \sum_{\mathcal{R} \in S} \; \sum_{(\mathcal{R}, (u_m), (y_m)) \in \pi^{-1}( \mathcal{R} )} \; \sum_{(a,b) \in \varphi^{-1}( \mathcal{R}, (u_m), (y_m) )} \mu(a) \mu(b).
\end{align*}
Applying (\ref{Eqn:BlueFish}) to each term in the sum, we get
\begin{align*}
& \mu  \otimes \mu( D_{n,k}^j ) \\
&  \leq \bigl(K^{16} \bigr)^{k/n} \gamma^{2k-j-n}   \sum_{\mathcal{R} \in S} \; \sum_{(\mathcal{R}, (u_m), (y_m)) \in \pi^{-1}( \mathcal{R} )} \;  \prod_{m=1}^{N_1+1}  \mu( u_m) \prod_{m=1}^{N_2+1} \mu( y_m).
\end{align*} 
Then since the sum of $\mu(u_m)$ over any set of words $u_m$ of the same length is less than or equal to one, we see that
\begin{equation*} 
\mu \otimes \mu( D_{n,k}^j ) \leq  \bigl(K^{16} \bigr)^{k/n} \gamma^{2k-j-n} |S |.
\end{equation*}
Combining this estimate with the bound on $|S|$ established in the previous paragraph, we obtain 
\begin{equation*}
\mu \otimes \mu( D_{n,k}^j ) \leq  \bigl(2^{16}K^{16}k^{16} \bigr)^{k/n} \gamma^{2k-j-n} .
\end{equation*}

Recall that $k = o(n^2/\log(n))$. Then for all large enough $n$, we have $k \leq n^2$.
Let $p_2(x) = (2K)^{16} x^{32}$. 
Then by the previous display, for all large enough $n$, we obtain that  $\mu \otimes \mu( D_{n,k}^j ) \leq  p_2(n)^{k/n} \gamma^{2k-j-n}$.
\end{proof}

\begin{lemma} \label{Lemma:Love}
There exists a polynomial $p_3(x)$ such that for all large enough $n$, 
\begin{equation*}
\mu \otimes \mu( Q_{n,k} ) \leq p_3(n) e^{-2n (h(\mu) - \delta)} \gamma^n.
\end{equation*}
\end{lemma}
\begin{proof}
Consider $n \geq n_0$. First define the set $F = [0,k) \sqcup [k+1,2k]$. 
Then we define a map $\varphi : Q_{n,k} \to \{ (\mathcal{R}, w) : \mathcal{J} \subset \mathcal{C}_n(S) \times \mathcal{C}_n(S), w \in \mathcal{A}^{ S \setminus A(\mathcal{R}) } \}$ as follows.
Let $(a,b) \in Q_{n,k}$. We let $a \sqcup b$ denote the element of $\mathcal{A}^F$ such that $(a \sqcup b)|_{[0,k)} = a$ and $(a \sqcup b)|_{[k+1,2k]} = b$. Since $W_n(a) \cap W_n(b) \neq \varnothing$, there exists $I \in \mathcal{C}_n([0,k))$ and $J \in \mathcal{C}_n([k+1,2k])$ such that $(a \sqcup b)|_I = (a \sqcup b)|_J$, and we assume that $J$ is (lexicographically) minimal among all such intervals. We partition $Q_{n,k}$ into $\tilde{Q}_{n,k}$ and $\hat{Q}_{n,k}$, where $\tilde{Q}_{n,k}$ consists of all pairs $(a,b)$ such that $J \cap [2\ell-n,2k] = \varnothing$, and $\hat{Q}_{n,k}$ contains the remaining pairs. Our definition of $\varphi(a,b)$ will depend on whether $(a,b)$ is in $\tilde{Q}_{n,k}$ or $\hat{Q}_{n,k}$.

First suppose that $(a,b) \in \tilde{Q}_{n,k}$. Let $\mathcal{R} \subset \mathcal{C}_n(F) \times \mathcal{C}_n(F)$ be the following set containing three pairs of intervals:  $\{ ( [0,n), [\ell-1,k) ), \, (I,J), \, ([k+1,k+n], [2k-n+1, 2k]) \}$. 
Note that the block decomposition of $a$ induced by $A(\mathcal{R}) \cap [0,k)$ has the form  $\bigl( u_1, v_1 \bigr)$, where $v_1 = a|_{[\ell-1,k)}$. Similarly, the block decomposition of $b$ induced by $A(\mathcal{R}) \cap [k+1,2k]$ has the form  $\bigl( (y_m)_{m=1}^2, (z_m)_{m=1}^{2} \bigr)$, where $z_1 = (a \sqcup b)|_{J}$ and $z_2 = (a \sqcup b)|_{[2k-n+1,2k]}$. Finally, we define $\varphi(a,b) = (\mathcal{R}, u_1, (y_m)_{m=1}^2 )$. Furthermore, we note that
\begin{align*}
\mu(a) & = \mu(u_1 v_1)  \leq K \mu(u_1) \mu(v_1),
\end{align*}
and
\begin{align*}
\mu(b) & = \mu( y_1 z_1 y_2 z_2)  \leq K^3 \mu(y_1) \mu(y_2) \mu(z_1) \mu(z_2).
\end{align*}
Since $a,b \in G_{n,k}$, we must have that $v_1, z_2 \in E_n$, and therefore $\mu(v_1) \leq e^{-(h(\mu)-\delta)n}$ and $\mu(z_2) \leq e^{-(h(\mu)-\delta)n}$. Also,  since $z_1$ has length $n \geq n_0$, we have that $\mu(z_1) \leq \gamma^n$. Putting these estimates together, we obtain
\begin{equation} \label{Eqn:Sonnet}
\mu(a) \mu(b) \leq K^4 e^{-2n(h(\mu)-\delta)} \gamma^n \mu(u_1) \mu(y_1) \mu(y_2).
\end{equation}

Now suppose that $(a,b) \in \hat{Q}_{n,k}$. In this case we let $\mathcal{R} \subset \mathcal{C}_n(F) \times \mathcal{C}_n(F)$ be a different set of three pairs of intervals: $\mathcal{R} = \{([0,n), \, [\ell-1,k)),(I,J), \, ([2k-n+1, 2k], [k+1,k+n])\}$. Note that the third pair of intervals listed is not in lexicographical order. Let $\bigl( u_1, v_1 \bigr)$ be the repeat block decomposition of $a$ induced by $A(\mathcal{R}) \cap [0,k)$, and let $\bigl( (z_m)_{m=1}^2, (y_m)_{m=1}^{2} \bigr)$ be the block decomposition induced by $A(\mathcal{R})|_{[k+1,2k]}$, by which we mean that $(a \sqcup b)|_{[k+1,2k]} = z_1 y_1 z_2 y_2$, where $z_1$ and $z_2$ have length $n$. In this case, we define $\varphi(a,b) = (\mathcal{R}, u_1, (y_m)_{m=1}^2 )$. Note that
\begin{align*}
\mu(a) & = \mu(u_1 v_1)  \leq K \mu(u_1) \mu(v_1),
\end{align*}
and
\begin{align*}
\mu(b) & = \mu(z_1 y_1 z_2 y_2)  \leq K^3 \mu(y_1) \mu(y_2) \mu(z_1) \mu(z_2).
\end{align*}
Since $a,b \in G_{n,k}$, we have that $v_1, z_1 \in E_n$, and thus $\mu(v_1) \leq e^{-(h(\mu)-\delta)n}$ and $\mu(z_1) \leq e^{-(h(\mu)-\delta)n}$. Also, since $z_2$ has length $n \geq n_0$,  we have that $\mu(z_2) \leq \gamma^n$. Combining these estimates, we see that
\begin{equation} \label{Eqn:Ode}
\mu(a) \mu(b) \leq K^4 e^{-2n(h(\mu)-\delta)} \gamma^n \mu(u_1) \mu(y_1) \mu(y_2).
\end{equation}

Now define the projection map $\pi : \varphi(Q_{n,k}) \to \{ \mathcal{R} : \mathcal{J} \subset \mathcal{C}_n(F) \times \mathcal{C}_n(F) \}$, given by $\pi(\mathcal{R},w) = \mathcal{R}$. Let $S = \pi \circ \varphi(Q_{n,k})$. Since $|\mathcal{C}_n(F)| \leq (2k)$, we get $|\mathcal{C}_n(F) \times \mathcal{C}_n(F)| \leq (2k)^2$. Moreover, since each $\mathcal{R}$ in $\pi \circ \varphi( Q_{n,k})$ satisfies $|\mathcal{R}| = 3$, we get $|S| = |\pi \circ \varphi(Q_{n,k})| \leq |\mathcal{C}_n(F) \times \mathcal{C}_n(F)|^3 \leq (2k)^6$. 

By rearranging the sum, we find
\begin{align*}
\mu \otimes \mu( Q_{n,k} ) &  = \sum_{(a,b) \in Q_{n,k}} \mu(a)\mu(b) \\
& = \sum_{ \mathcal{R} \in S} \; \sum_{(\mathcal{R},u_1,(y_m)_{m=1}^2) \in \pi^{-1}(\mathcal{R})} \; \sum_{(a,b) \in \varphi^{-1} (\mathcal{R}, u_1, (y_m)_{m=1}^2)} \mu(a) \mu(b).
\end{align*}
Then by applying the estimates (\ref{Eqn:Sonnet}) and (\ref{Eqn:Ode}) to each term, we get
\begin{align*}
\mu \otimes \mu( Q_{n,k} ) & \leq  \sum_{ \mathcal{R} \in S} \; \sum_{(\mathcal{R},u_1,(y_m)_{m=1}^2) \in \pi^{-1}(\mathcal{R})}  K^4 e^{-2n(h(\mu)-\delta)} \gamma^n \mu(u_1) \mu(y_1) \mu(y_2).
\end{align*}
Summing over all $u_1$, $y_1$, and $y_2$, we obtain
\begin{equation*}
\mu \otimes \mu( Q_{n,k} )  \leq    K^4 e^{-2n(h(\mu)-\delta)} \gamma^n |S| \leq K^4 (2k)^6  e^{-2n(h(\mu)-\delta)} \gamma^n,
\end{equation*}
where the second inequality uses the bound on $|S|$ established in the previous paragraph

Recall that $k = o(n^2/\log(n))$, and hence for all large enough $n$, we have $k \leq n^2$. 
Let $p_3(x) = 2^6 K^4 x^{12}$.
Then by the previous inequality, for all large enough $n$, we see that $\mu \otimes \mu( Q_{n,k} ) \leq p_3(n) e^{-2n(h(\mu)-\delta)} \gamma^n$.
\end{proof}

\begin{lemma} \label{Lemma:Peace}
There exists a polynomial $p_4(x)$ such that for all large enough $n$ and $1 \leq j \leq 2 \ell$,
\begin{equation*}
\mu \otimes \mu( Q_{n,k}^j ) \leq p_4(k)^{k/n} e^{- 2n (h(\mu) - \delta)} \gamma^{2\ell - j}.
\end{equation*} 
\end{lemma}
\begin{proof}
Consider $n \geq n_0$ and $1 \leq j \leq 2\ell$.
Let $F = [0,k) \cup [k+1,2k]$. We begin by defining a map $\varphi : Q_{n,k}^j \to \{ (\mathcal{R}, w) : \mathcal{R} \subset \mathcal{C}_n(S) \times \mathcal{C}_n(S), w \in \mathcal{A}^{F \setminus A(\mathcal{R})} \}$ as follows.
Let $(a,b) \in Q_{n,k}^j$. We let $(a \sqcup b)$ denote the element of $\mathcal{A}^F$ such that $(a \sqcup b)|_{[0,k)} = a$ and $(a \sqcup b)|_{[k+1,2k]} = b$. Let $\mathcal{R}$ be a repeat cover for $a \sqcup b$ such that $|\mathcal{R}| \leq 4 |F|/n = 8k/n$, which exists by Lemma \ref{Lemma:EfficientRepeatCovers}. 
Then let $\bigl( (u_m)_{m=1}^{N_1+1}, (v_m)_{m=1}^{N_1} \bigr)$ be the repeat block decomposition of $a$ induced by the set $A(\mathcal{R}) \cap [0,k-n)$, and let $\bigl( (y_m)_{m=1}^{N_2+1}, (z_m)_{m=1}^{N_2} \bigr)$ be the repeat block decomposition for $b$ induced by the set $A(\mathcal{R}) \cap [k+1,2k-n]$. Additionally, let $v_{N_1+1} = a_{[k-n,k)}$ and $z_{N_2+1} = b_{[2k-n+1,2k]}$. 
Set $\varphi(a,b) = (\mathcal{R}, (u_m)_{m=1}^{N_1+1}, (y_m)_{m=1}^{N_2+1} )$.
 Furthermore, by Lemma \ref{Lemma:Birthday}, we have that
\begin{equation*}
\mu(a) \leq K^{2N_1+1} \prod_{m=1}^{N_1+1} \mu(v_m) \prod_{m=1}^{N_1+1} \mu( u_m),
\end{equation*}
and
\begin{equation*}
\mu(b) \leq K^{2N_2+1} \prod_{m=1}^{N_2+1} \mu(z_m) \prod_{m=1}^{N_2+1} \mu( y_m).
\end{equation*}
Since $v_{N_1+1}, z_{N_2+1} \in E_n$, we have $\mu(v_{N_1+1}) \leq e^{-n(h(\mu)-\delta)}$ and $\mu(z_{N_2+1}) \leq e^{-n(h(\mu)-\delta)}$. Also, for $m = 1, \dots, N_1 -1$, the length of $v_m$ is at least $n$, and we get $\mu(v_m) \leq \gamma^{|v_m|}$. For $v_{N_1}$, we always have $\mu(v_{N_1}) \leq \gamma^{|v_{N_1}|-n_0}$. Similarly, for $m = 1, \dots, N_2-1$, the length of $z_m$ is at least $n$, and we get $\mu(z_m) \leq \gamma^{|z_m|}$. As for $z_{N_2}$, we always have $\mu(z_{N_2}) \leq \gamma^{|z_{N_2}|-n_0}$. Then for large enough $n$, we have
\begin{equation*}
\mu(a) \leq K^{2N_1+1} \gamma^{\sum_m |v_m| - n_0} e^{-n(h(\mu)-\delta)} \prod_{m=1}^{N_1+1} \mu( u_m),
\end{equation*}
and
\begin{equation*}
\mu(a) \leq K^{2N_2+1} \gamma^{\sum_m |z_m| - n_0} e^{-n(h(\mu)-\delta)} \prod_{m=1}^{N_2+1} \mu( y_m).
\end{equation*}
Using that $N_1+N_2 \leq |\mathcal{R}| \leq 8k/n$ and $\sum_m |v_m| + \sum_m |z_m| = |A(\mathcal{R})| \geq 2\ell-j - 2 $ (by Lemma \ref{Lemma:RepeatArea}), we obtain
\begin{equation} \label{Eqn:RedFish}
\mu(a) \mu(b) \leq \bigl(K^{32} \bigr)^{k/n} \gamma^{2\ell - j - 2 - 2n_0} e^{-2n(h(\mu)-\delta)} \prod_{m=1}^{N_1+1} \mu( u_m) \prod_{m=1}^{N_2+1} \mu( y_m).
\end{equation}

Now we define the projection map $\pi : \varphi(Q_{n,k}^j) \to \{ \mathcal{R} : \mathcal{R} \subset \mathcal{C}_n(F) \times \mathcal{C}_n(F) \}$, given by $\pi( \mathcal{R}, (u_m), (y_m) ) = \mathcal{R}$. Let $S = \pi \circ \varphi(Q_{n,k}^j)$. Since $|\mathcal{C}_n(F)| \leq 2k$, we see that $|\mathcal{C}_n(F) \times \mathcal{C}_n(F)| \leq (2k)^2$. Moreover, since each $\mathcal{R}$ in $S$ satisfies $|\mathcal{R}| \leq 8k/n$, we estimate $|S| = |\pi \circ \varphi( Q_{n,k}^j )| \leq | \mathcal{C}_n(F) \times \mathcal{C}_n(F)|^{8k/n} \leq (2^{16} k^{16})^{k/n}$.

Let us now estimate $\mu \otimes \mu( Q_{n,k}^j)$. 
By rearranging the sum, we get
\begin{align*}
\mu \otimes \mu( Q_{n,k}^j ) & = \sum_{(a,b) \in Q_{n,k}^j} \mu(a) \mu(b) \\
& = \sum_{\mathcal{R} \in S} \; \sum_{(\mathcal{R}, (u_m), (y_m)) \in \pi^{-1}( \mathcal{R} )} \; \sum_{(a,b) \in \varphi^{-1}( \mathcal{R}, (u_m), (y_m) )} \mu(a) \mu(b).
\end{align*}
By applying (\ref{Eqn:RedFish}) to each term in the sum, we see that
\begin{align*}
& \mu  \otimes \mu( Q_{n,k}^j ) \\
&  \leq  \sum_{\mathcal{R} \in S} \; \sum_{(\mathcal{R}, (u_m), (y_m)) \in \pi^{-1}( \mathcal{R} )}   (K^{32} \gamma^{-2n_0-2})^{k/n} \gamma^{2\ell-j} e^{-n(h(\mu)-\delta)} \prod_{m=1}^{N_1+1} \mu( u_m) \prod_{m=1}^{N_2+1} \mu( y_m),
\end{align*} 
and then summing over all $u_m$ and $y_m$ gives
\begin{equation} \label{Eqn:Ecke}
\mu \otimes \mu( Q_{n,k}^j ) \leq  (K^{32} \gamma^{-2n_0 - 2})^{k/n} e^{-n(h(\mu)-\delta)} \gamma^{2\ell-j} | S |.
\end{equation}
Combining this estimate with the bound on $|S|$ from the previous paragraph, we obtain
\begin{equation*}
\mu \otimes \mu( Q_{n,k}^j ) \leq  \bigl(2^{32}K^{32} \gamma^{-2n_0-2}k^{16} \bigr)^{k/n} e^{-n(h(\mu)-\delta)} \gamma^{2\ell-j} .
\end{equation*}

Recall that  $k = o(n^2/\log(n))$, and hence for all large enough $n$, we have $k \leq n^2$.
Let $p_4(x) = (2K)^{32} \gamma^{-2n_0-2}x^{32}$.
Then by the previous displayed inequality, for all large enough $n$ and all $1 \leq j \leq 2\ell$, we have $\mu \otimes \mu( Q_{n,k}^j ) \leq  p_4(n)^{k/n} e^{-n(h(\mu)-\delta)} \gamma^{2\ell-j}$.
%
\end{proof}

\section{Moment bounds} \label{Sect:MomentBounds}

In this section we prove properties (\ref{Item:1})-(\ref{Item:4}) concerning the expectation and variance of $\phi_{n,k}$ and $\psi_{n,k}$, which are used in the proof of Theorem \ref{Thm:Pressure}. Throughout this section, we use the same environment (notation, parameters, and assumptions) as in the proof of Theorem \ref{Thm:Pressure}.

\begin{lemma} \label{Lemma:Skate}
For all $n \geq 1$, the expectation of $\phi_{n,k}$ satisfies
\begin{equation*}
\mathbb{E} \bigl[ \phi_{n,k} \bigr] \geq K^{-1} \alpha^{\ell} e^{Pk}.
\end{equation*}
Furthermore, 
\begin{equation*}
\lim_n \frac{1}{k} \log \mathbb{E}\bigl[ \phi_{n,k} \bigr] = P + \log( \alpha)
\end{equation*}
\end{lemma}
\begin{proof}
Let $n \geq 1$. Then by (\ref{Eqn:BasicExp}) and our choice of $K$, we have
\begin{align*}
\mathbb{E} \bigl[ \phi_{n,k} \bigr] 
 & = \sum_{u \in B_k(X)} e^{S_k f(u)} \alpha^{|W_n(u)|} \\
 & = \sum_{j = 1}^{\ell} \alpha^j \sum_{u \in B_{n,k}^j} e^{S_k f(u)}. \\
 & \geq K^{-1} e^{P k}  \sum_{j = 1}^{\ell} \alpha^j \mu \bigl(B_{n,k}^j \bigr).
\end{align*}
Using that $\alpha^j \geq \alpha^{\ell}$ for all $j \leq \ell$ and $\sum_j \mu(B_{n,k}^j) = \mu(B_n(X)) = 1$, we get
\begin{align*}
\mathbb{E} \bigl[ \phi_{n,k} \bigr]  \geq K^{-1} \alpha^{\ell} e^{Pk},
\end{align*}
which establishes the first conclusion of the lemma.

Now we consider letting $n$ tend to infinity.
By the first conclusion of the lemma, we have that
\begin{equation*}
\liminf_n \frac{1}{k} \log \mathbb{E} \bigl[ \phi_{n,k} \bigr] \geq  P + \log \alpha.
\end{equation*}
Also, for any $n$, our choice of $K$ yields 
\begin{equation*}
\mathbb{E} \bigl[ \phi_{n,k} \bigr]  = \sum_{j=1}^{k-n+1} \alpha^j \sum_{u \in B_{n,k}^j} e^{S_k f(w)} \leq K e^{P k} \sum_{j=1}^{k-n+1} \alpha^j \mu(B_{n,k}^j).
\end{equation*}

Recall that $\alpha > \gamma$, and therefore $\alpha \gamma^{-1} > 1$.
By Lemma \ref{Lemma:NPS}, there exists a polynomial $p_1(x)$ such that for large enough $n$, we have
\begin{align*}
\mathbb{E} \bigl[ \phi_{n,k} \bigr] & \leq K e^{P k} \sum_{j=1}^{\ell} \alpha^j \mu(B_{n,k}^j) \\
& \leq K e^{P k} \sum_{j=1}^{\ell} \alpha^j p_1(n)^{k/n} \gamma^{k-j} \\
& \leq K e^{Pk} p_1(n)^{k/n} \gamma^k \sum_{j=1}^{\ell} (\alpha \gamma^{-1})^j \\
& \leq K e^{Pk} p_1(n)^{k/n} \gamma^k (\alpha \gamma^{-1})^{\ell} \frac{1}{1-(\alpha^{-1} \gamma)} \\
& = K e^{Pk} p_1(n)^{k/n} \alpha^{\ell} \gamma^{n} \frac{1}{1-(\alpha^{-1} \gamma)}.
\end{align*}
Since $n/k \to 0$ and $n^{-1} \log p_1(n) \to 0$, we obtain that
\begin{equation*}
\limsup_n \frac{1}{k} \log \mathbb{E} \bigl[ \phi_{n,k} \bigr] \leq P + \log \alpha,
\end{equation*}
which finishes the proof.
\end{proof}

\begin{lemma} \label{Lemma:Board}
For all $n \geq 1$, the expectation of $\psi_{n,k}$ satsifies
\begin{equation*}
\mathbb{E} \bigl[ \psi_{n,k} \bigr] \geq |E_n|^{-1} K^{-1} \alpha^{\ell} e^{Pk} \mu \bigl( G_{n,k} \bigr).
\end{equation*}
Furthermore,
\begin{equation*}
\lim_n \frac{1}{k} \log \mathbb{E}\bigl[ \psi_{n,k} \bigr] = P + \log( \alpha).
\end{equation*}
\end{lemma}
\begin{proof}
Let $n \geq 1$. Then by (\ref{Eqn:BasicExp}) and our choice of $K$, we have
\begin{align*}
\mathbb{E} \bigl[ \psi_{n,k} \bigr]  
& = \frac{1}{|E_n|} \sum_{u \in G_{n,k}} e^{S_k f(u)} \alpha^{|W_n(u)|} \\
 & =  \frac{1}{|E_n|} \sum_{j = 1}^{\ell} \alpha^j \sum_{u \in G_{n,k}^j} e^{S_k f(u)} \\
 & \geq |E_n|^{-1} K^{-1} e^{Pk} \sum_{j = 1}^{\ell} \alpha^j \mu \bigl( G_{n,k}^j \bigr).
\end{align*}
Since $\alpha^j \geq \alpha^{\ell}$ for all $j \leq \ell$ and $\sum_j \mu(G_{n,k}^j) = \mu(G_{n,k})$, we get
\begin{align*}
\mathbb{E} \bigl[ \psi_{n,k} \bigr]    \geq |E_n|^{-1} K^{-1} \alpha^{\ell} e^{Pk} \mu \bigl(G_{n,k} \bigr),
\end{align*}
which establishes the first conclusion of the lemma. 

Now we consider letting $n$ tend to infinity. 
By the first conclusion of the lemma, we have that
\begin{align*}
\liminf_n \frac{1}{k} \log \mathbb{E} \bigl[ \psi_{n,k} \bigr] & \geq  P + \log \alpha + \liminf_n \frac{1}{k} \log |E_n|^{-1} + \frac{1}{k} \log \mu \bigl( G_{n,k} \bigr). 
\end{align*}
Note that $|E_n| \leq |\mathcal{A}|^n$, and by Lemma \ref{Lemma:NewAsia}, for large enough $n$, we have $\mu(G_{n,k}) \geq 2^{-1} K^{-1} e^{- n (h(\mu)+\delta)}$. Therefore
\begin{equation*}
\liminf_n \frac{1}{k} \log \mathbb{E} \bigl[ \psi_{n,k} \bigr]  \geq  P + \log \alpha  + \liminf_n \biggl( \frac{n}{k} \log |\mathcal{A}|^{-1} + \frac{n}{k} \log (e^{-(h(\mu)+\delta)}) \biggr).
\end{equation*}
Finally, using that $n/k \to 0$, we obtain
\begin{equation*}
\liminf_n \frac{1}{k} \log \mathbb{E} \bigl[ \psi_{n,k} \bigr]  \geq  P + \log \alpha.
\end{equation*}
Also, by Lemma \ref{Lemma:Skate} and the fact that $\psi_{n,k} \leq \phi_{n,k}$, we have
\begin{equation*}
\limsup_n \frac{1}{k} \log \mathbb{E} \bigl[ \psi_{n,k} \bigr] \leq \limsup_n \frac{1}{k} \log \mathbb{E} \bigl[ \phi_{n,k} \bigr] \leq P + \log \alpha.
\end{equation*}
Taken together, the previous two inequalities yield that
\begin{equation*}
\lim_n  \frac{1}{k} \log \mathbb{E} \bigl[ \psi_{n,k} \bigr] = P + \log \alpha,
\end{equation*}
as desired.
\end{proof}

\begin{lemma} \label{Lemma:Roller}
There exists $\rho_1 >0$ such that for all large enough $n$, 
\begin{equation*}
\frac{ \Var \bigl[ \phi_{n,k} \bigr] }{ \mathbb{E} \bigl[ \phi_{n,k} \bigr]^2 } \leq e^{- \rho_1 n};
\end{equation*}
\end{lemma}
\begin{proof}
Using the fact that the variance of a sum is the sum of the covariances,  (\ref{Eqn:BasicVar}), and our choice of $K$, we have
\begin{align*}
\Var \bigl[ \phi_{n,k} \bigr] & = \sum_{u,v \in B_k(X)} \alpha^{|W_n(u) \cup W_n(v)|} \Bigl(1 - \alpha^{|W_n(u) \cap W_n(v)|} \Bigr) e^{S_k f(u) + S_k f(v)} \\
 & \leq \sum_{j=1}^{2\ell-1} \alpha^j \sum_{(u,v) \in D_{n,k}^j} e^{S_k f(u) + S_k f(v)} \\
 & \leq K^2 e^{2Pk}  \sum_{j=1}^{2\ell-1} \alpha^j \mu \otimes \mu \bigl( D_{n,k}^j \bigr).
\end{align*}

Let $C = 1/(1- (\alpha^{-1} \gamma))$. 
Then by the lower bound on $\mathbb{E}[ \phi_{n,k} ]$ from Lemma \ref{Lemma:Skate} and the upper bound on $\mu \otimes \mu(D_{n,k}^j)$ from Lemma \ref{Lemma:Super}, there exists a polynomial $p_2(x)$ such that for large enough $n$, we have
\begin{align*}
\frac{\Var \bigl[ \phi_{n,k} \bigr] }{ \mathbb{E} \bigl[ \phi_{n,k} \bigr] ^2 } & \leq \frac{ K^2e^{2 P k} \sum_{j=1}^{2\ell-1} \alpha^j \mu \otimes \mu \bigl( D_{n,k}^j \bigr) }{ K^{-2} \alpha^{2\ell} e^{2Pk}  } \\
& = K^4 \alpha^{-2\ell} \sum_{j = 1}^{2\ell-1} \alpha^j  \mu \otimes \mu \bigl( D_{n,k}^j \bigr) \\
 & \leq K^4 \alpha^{-2\ell} \sum_{j = 1}^{2\ell-1} \alpha^j  p_2(n)^{k/n} \gamma^{2\ell+n - j } \\
 & \leq K^4 p_2(n)^{k/n} \alpha^{-2\ell} \gamma^{2\ell+n}   \sum_{j = 1}^{2\ell-1} (\alpha \gamma^{ - 1})^j \\
 & \leq K^4 p_2(n)^{k/n}  \alpha^{-2\ell}  \gamma^{2\ell+n} (\alpha \gamma^{-1})^{2\ell} C.
 \end{align*}
 Rewriting this estimate, we find
\begin{align*}
 \frac{\Var \bigl[ \phi_{n,k} \bigr] }{ \mathbb{E} \bigl[ \phi_{n,k} \bigr] ^2 } & = \exp \biggl( 4 \log K + \log C + \frac{k}{n} \log p_2(n) + n \log \gamma \biggr) \\
 & \leq \exp \biggl( n \Bigl( q \frac{k \log n}{n^2} + \log \gamma + \frac{4}{n} \log K  + \frac{1}{n}\log C \Bigr) \biggr),
\end{align*}
where $p_2(x) \leq x^q$ for all large enough $x$.
Since $k = o( n^2 / \log n)$ and $\log \gamma < 0$, we obtain the desired bound.
\end{proof}

\begin{lemma} \label{Lemma:Blade}
There exists $\rho_2 >0$ such that for all large enough $n$, 
\begin{equation*}
\frac{ \Var \bigl[ \psi_{n,k} \bigr] }{ \mathbb{E} \bigl[ \psi_{n,k} \bigr]^2 } \leq e^{- \rho_2 n}.
\end{equation*}
\end{lemma}
\begin{proof}
Let $b = 2\ell-n$. Using the fact that the variance of a sum is the sum of the covariances, (\ref{Eqn:BasicVar}), and our choice of $K$, we have
\begin{align*}
\Var \bigl[ \psi_{n,k} \bigr] & = \frac{1}{|E_n|^2} \sum_{u,v \in G_{n,k}} \alpha^{|W_n(u) \cup W_n(v)|} \Bigl(1 - \alpha^{|W_n(u) \cap W_n(v)|} \Bigr) e^{S_k f(u) + S_k f(v)} \\
 & \leq \frac{1}{|E_n|^2} \sum_{j=1}^{2\ell} \alpha^j \sum_{(u,v) \in Q_{n,k}^j} e^{S_k f(u) + S_k f(v)} \\
 & \leq K^2 \frac{ e^{2Pk} }{|E_n|^2}  \sum_{j=1}^{2\ell} \alpha^j \mu \otimes \mu \bigl( Q_{n,k}^j \bigr).
\end{align*}
Dividing by $\mathbb{E}[ \psi_{n,k}]$ and using the lower bound on $\mathbb{E}[ \psi_{n,k}]$ in Lemma \ref{Lemma:Board} and the lower bound on $\mu(G_{n,k})$ in Lemma \ref{Lemma:NewAsia}, we see that
\begin{align*}
\frac{\Var \bigl[ \psi_{n,k} \bigr] }{ \mathbb{E} \bigl[ \psi_{n,k} \bigr] ^2 } & \leq \frac{ K^2 e^{2 P k} \sum_{j=1}^{2\ell-1} \alpha^j \mu \otimes \mu \bigl( Q_{n,k}^j \bigr) }{ K^{-2} \alpha^{2\ell} e^{2Pk} \mu(G_{n,k})^2  } \\
& \leq \frac{K^4 \alpha^{-2\ell} \sum_{j = 1}^{2\ell-1} \alpha^j  \mu \otimes \mu \bigl( Q_{n,k}^j \bigr)}{ (2K)^{-2} e^{-2n(h(\mu)+\delta)}} \\
& = (2K)^6 \alpha^{-2\ell} e^{2n(h(\mu)+\delta)} \Biggl( \sum_{j = 1}^{b-1} \alpha^j  \mu \otimes \mu \bigl( Q_{n,k}^j \bigr) + \sum_{j=b}^{2\ell-1} \alpha^j  \mu \otimes \mu \bigl( Q_{n,k}^j \bigr) \Biggr).
\end{align*}
Let $C = 1/(1 - (\alpha^{-1} \gamma))$. 
Note that $\alpha^j \leq \alpha^b$ for $j \geq b$. Applying this fact and the upper bounds on $\mu \otimes \mu(Q_{n,k})$ and $\mu \otimes \mu (Q_{n,k}^j)$ from Lemmas \ref{Lemma:Love} and \ref{Lemma:Peace}, respectively, we get that there are polynomials $p_3(x)$ and $p_4(x)$ such that for all large enough $n$,
\begin{align*}
\frac{\Var \bigl[ \psi_{n,k} \bigr] }{ \mathbb{E} \bigl[ \psi_{n,k} \bigr] ^2 } & 
  \leq (2K)^6 \alpha^{-2\ell} e^{2n(h(\mu)+\delta)} \Biggl( \sum_{j = 1}^{b-1} \alpha^j  p_4(n)^{k/n} e^{- 2n (h(\mu) - \delta)} \gamma^{2\ell - j} + \alpha^{b} p_3(n) e^{-2n (h(\mu) - \delta)} \gamma^n \Biggr) \\
 & \leq (2K)^6 \alpha^{-2 \ell} e^{4n \delta} \Biggl( p_4(n)^{k/n} \gamma^{2\ell} \sum_{j=1}^{b-1} (\alpha \gamma^{-1})^j + \alpha^{b} \gamma^n p_3(n) \Biggr) \\
 & \leq (2K)^6 \alpha^{-2 \ell} e^{4n \delta} \biggl( p_4(n)^{k/n} \gamma^{2\ell} C (\alpha \gamma^{-1})^b + \alpha^{b} \gamma^n p_3(n) \biggr) \\
 & = (2K)^6 e^{4n \delta} \gamma^n \alpha^{-n} \Bigl( C p_4(n)^{k/n} + p_3(n) \Bigr).
 \end{align*}
 Rewriting this estimate, we have
 \begin{equation*}
 \frac{\Var \bigl[ \psi_{n,k} \bigr] }{ \mathbb{E} \bigl[ \psi_{n,k} \bigr] ^2 } \leq \exp \Biggl( n \biggl( \log(\gamma \alpha^{-1}) + 4 \delta + \frac{6}{n} \log(2K) + \frac{1}{n} \log \bigl( C p_4(n)^{k/n} + p_3(n)  \bigr) \biggr) \Biggr).
\end{equation*}
Let $q >1$ be such that for all large enough $x$, we have $C p_4(x)^{k/n} + p_3(x) \leq x^{qk/n}$.
Then for all large enough $n$, we get
\begin{equation*}
\frac{\Var \bigl[ \psi_{n,k} \bigr] }{ \mathbb{E} \bigl[ \psi_{n,k} \bigr] ^2 } \leq \exp \Biggl( n \biggl( \log(\gamma \alpha^{-1}) + 4 \delta + \frac{6}{n} \log(2K) + \frac{q k}{n^2} \log n  \bigr) \biggr) \Biggr).
\end{equation*}
Since $k = o(n^2/\log(n))$, and $\log(\gamma \alpha^{-1}) + 4\delta < 0$ (by our choice of $\delta$ in the proof of Theorem \ref{Thm:Pressure}), we obtain the desired bound.
\end{proof}

\section{Bounds on pressure} \label{Sect:UBandLB}

In this section we work with the same notation, parameters, and assumptions as in the proof of Theorem \ref{Thm:Pressure}.

\begin{lemma} \label{Lemma:UB}
For each $n$,
\begin{equation*}
P_{Y_n}(f) \leq \frac{1}{k} \log \phi_{n,k}.
\end{equation*}
\end{lemma}
\begin{proof}
By subadditivity in the definition of pressure, we have that for all $m \geq 1$,
\begin{align*}
P_{Y_n}(f) & \leq \frac{1}{m} \log \sum_{u \in B_m(Y_n)} e^{S_m f(u)}.
\end{align*}
We apply this inequality with $m = k$.
Also, since $B_k(Y_n) \subset \{ u \in B_k(X) : \xi_u = 1\}$, we have
\begin{align*}
 \frac{1}{k} \log \sum_{u \in B_k(Y_n)} e^{S_m f(u)} & \leq \frac{1}{k} \log \sum_{u \in B_k(X)} e^{S_k f(u)} \xi_u \\
  & = \frac{1}{k} \log \phi_{n,k}.
\end{align*}
Combining the two previous inequalities yields the desired conclusion.
\end{proof}

\begin{lemma} \label{Lemma:LB}
For any $\epsilon >0$, for all large enough $n$,
\begin{equation*}
\frac{1}{k} \log \psi_{n,k} - \epsilon/2 \leq P_{Y_n}(f).
\end{equation*}
\end{lemma} 
\begin{proof}
Let $\mathcal{F} = \mathcal{F}_n$ and $Y = Y_n$. For $v \in E_n$, and $m \geq n$, we let
\begin{equation*}
Z_m(v) = \biggl\{ u \in B_m(X) : W_n(u) \cap \mathcal{F} = \varnothing, \text{ and } \forall q \in \{0,\dots,\lfloor m/\ell-1\rfloor\}, \, u_{q \ell+1}^{q\ell+n} = v \biggr\}.
\end{equation*}
Note that $\psi_{n,k}$ may be viewed as an average over the set $E_n$:
\begin{equation*}
\psi_{n,k} = \frac{1}{|E_n|} \sum_{v \in E_n} \sum_{u \in Z_k(v)} e^{S_kf(u)}.
\end{equation*}
Since the average over a finite set is always less than or equal to the maximum, there exists $v \in E_n$ such that
\begin{equation*}
\psi_{n,k} \leq \sum_{ u \in Z_k(v) } e^{S_k f(u) }.
\end{equation*}
For the sake of this proof, if $u \in B_m(X)$, then we let $\underline{S}_m f(u) = \inf_{x \in [u]} \sum_{j=0}^{m-1} f \circ \sigma^{j}(x)$. 
Observe that elements of $Z_{\ell}(v)$ can be arbitrarily concatenated to form words in $Y$. Hence, for any $q \in \N$, we note that $Z_{q \ell}(v) \subset B_{q\ell}(Y)$, and then we have
\begin{align*}
\sum_{ u \in B_{q\ell}(Y) } e^{S_{q\ell} f(u)} & \geq \sum_{u \in Z_{q \ell}(v)} e^{S_{q \ell} f(u)} \\
& \geq \sum_{u \in Z_{q \ell}(v)} e^{\underline{S}_{q \ell} f(u)}.
\end{align*}
Then by our choice of $K$,  we get 
\begin{align*}
\sum_{ u \in B_{q\ell}(Y) } e^{S_{q\ell} f(u)} & \geq  K^{-q} \sum_{u_0 \dots u_{q-1} \in Z_{q \ell}(v)} e^{\sum_{i =0}^{q-1} \underline{S}_{\ell} f(u_i)} \\
&  = K^{-q} \sum_{u_0 \in Z_{\ell}(v)} \cdots \sum_{u_{q-1} \in Z_{\ell}(v)} e^{\sum_{i =0}^{q-1} \underline{S}_{\ell} f(u_i)} \\
& = K^{-q} \biggl( \sum_{u \in Z_{\ell}(v)} e^{\underline{S}_{\ell}(u)} \biggr)^q \\
& \geq K^{-2q} \biggl( \sum_{u \in Z_{\ell}(v)} e^{S_{\ell}(u)} \biggr)^q \\
& \geq K^{-2q} \biggl( \sum_{u \in Z_{k}(v)} e^{S_{k}(u)} \biggr)^q e^{-\| f \|_{\infty} n q},
\end{align*}
where $\| f \|_{\infty} = \sup_{x \in X} |f(x)|$.
Now take logarithm, divide by $q \ell$, and let $q$ tend to infinity:
\begin{align*}
P_Y(f) \geq \frac{1}{\ell} \log \sum_{ u \in Z_k(v) } e^{S_k f(u) } - \frac{2 q \log K}{\ell} - \| f \|_{\infty} \frac{n}{\ell}.
\end{align*}
Then
\begin{equation*}
P_Y(f) \geq \frac{1}{k} \log \psi_{n,k} - \frac{2 q \log K}{\ell} - \| f \|_{\infty} \frac{n}{\ell}.
\end{equation*}
Finally, since $n/\ell \to 0$, we may choose $n$ large enough that 
\begin{equation*}
\frac{2 q \log K}{\ell} + \| f \|_{\infty} \frac{n}{\ell} < \epsilon/2,
\end{equation*}
which finishes the proof of the lemma.
\end{proof}

\section{Connection between pressure and escape rate} \label{Sect:Connections}

Here we relate the notions of pressure and escape rate. For a hole $H$ in an SFT $X$, we define the \textit{survivor set} to be the set of points that never fall into the hole (in either forward or backward time):
\begin{equation*}
Y = X \setminus \Biggl( \bigcup_{m \in \mathbb{Z}} \sigma^{-m}(H) \Biggr). 
\end{equation*}
For an SFT $(X,\sigma)$, a hole $H$ consisting of a finite union of cylinder sets, and an equilibrium state $\mu$ associated to a H\"{o}lder continuous potential function $f$, the following proposition relates the escape rate of $\mu$ through the hole $H$ to the pressure of $f$ on the survivor set $Y$. Although various versions of this result appear to be well-known (see, \textit{e.g.}, \cite{CM1,CMS1997}), we could not find an explicit reference for it, and we include a proof for completeness. For analogous results in various smooth settings, see the discussion of the \textit{escape rate formula} in \cite{BDM2010} and references therein.

\begin{proposition} \label{Prop:Crayola}
Let $X$ be a non-trivial mixing SFT, $f : X \to \R$ a H\"{o}lder continuous potential, and $\mu$ the Gibbs measure associated to $f$. Further, let $H$ be a finite union of cylinder sets in $X$, and let $Y$ be the survivor set of the open system $(X,\sigma,H)$. Then
\begin{equation*}
- \varrho(\mu : H) = P_X(f) - P_Y(f). 
\end{equation*}
\end{proposition}
\begin{proof}
Let $K$ satisfy the conclusions of Lemma \ref{Lemma:Oscar} for $X$, $f$, and $\mu$.
We suppose without loss of generality that $H$ is the union of cylinder sets corresponding to words of length $n$.
For $k \geq n$, let $B_k(X,H)$ denote the set of $w \in B_k(X)$ such that $w$ contains no subword in $H$, and let $P = P_X(f)$.  Let $\ell = \ell(k) = k - n+1$. Recall that
\begin{equation*}
M_{\ell} = \Bigl\{ x \in X : \forall j \in \{0,\dots,\ell-1\}, \, \sigma^j(x) \notin H \Bigr\},
\end{equation*}
so that we have $\mu(M_{\ell}) = \mu( B_k(X,H) )$. Note that since $n$ is fixed in this context, we have $\lim_{k \to \infty} \ell/k = 1$.

By our choice of $K$ and the fact that $B_k(Y) \subset B_k(X,H)$, we have that
\begin{align*}
 \mu(B_k(X,H)) & = \sum_{w \in B_k(X,H)} \mu(w) \\
 & \geq  \sum_{w \in B_k(X,H)} K^{-1} e^{-P k + S_k f(w)} \\
 & = K^{-1} e^{-P k} \sum_{w \in B_k(X,H)} e^{S_k f(w)} \\
 & \geq K^{-1} e^{-P k} \sum_{w \in B_k(Y)} e^{S_k f(w)}.
\end{align*}
It follows that
\begin{align*}
\frac{1}{k} \log \mu(B_k(X,H)) \geq - P + \frac{1}{k} \log \Lambda_k(Y) - \frac{1}{k} \log K,
\end{align*}
and letting $k$ tend to infinity, we see that
\begin{equation} \label{Eqn:ZiggyLower}
\liminf_{k} \frac{1}{k} \log \mu(M_{\ell}) \geq -P + \liminf_{k\to \infty} \frac{1}{k} \log \Lambda_k(Y) = - P + P_Y(f).
\end{equation}
Similarly, we have the following upper bound:
\begin{equation} \label{Eqn:ZiggyUpper}
\limsup_k \frac{1}{k} \log \mu(M_{\ell}) \leq -P + \limsup_{k \to \infty} \frac{1}{k} \log \Biggl( \sum_{w \in B_k(X,H)} e^{S_k f(w)} \Biggr).
\end{equation} 

Comparing the bounds in (\ref{Eqn:ZiggyLower}) and (\ref{Eqn:ZiggyUpper}) , we see that in order to finish the proof, it suffices to show that
\begin{equation} \label{Eqn:Sophie}
\limsup_{k \to \infty} \frac{1}{k} \log \Biggl( \sum_{w \in B_k(X,H)} e^{S_n f(w)} \Biggr) \leq P_Y(f).
\end{equation}
To get this inequality, we use ideas from \cite{LedWalters} to find an invariant measure $\nu$ supported on $Y$ such that $ h(\nu) + \int f d\nu$. The measure $\nu$ is obtained as follows.

For $k \geq 1$, suppose $B_k(X,H) = \{w^k_1, \dots, w^k_{m_k}\}$. Let $x^k_i$ be in $[w^k_i]$ such that $S_k f(w^k_i) = S_mf(x^k_i)$ (which exists by compactness and continuity). Then let
\begin{align*}
\mu_k & = \frac{ \sum_{j=1}^{m} e^{S_kf(x^k_j)} \delta_{x^k_j} }{  \sum_{j=1}^{m} e^{S_kf(x^k_j)} } \\
\nu_k & = \frac{1}{k} \sum_{j=0}^{k-1} S^j \mu_k.
\end{align*}
Since the space of Borel probability measures on $X$ is weak$^*$ compact, there is a subsequence $(k_j)$ such that $\nu_{k_j} \to \nu$ and along which the $\limsup$ in (\ref{Eqn:Sophie}) is obtained. Note that $\nu$ is in $M(X,S)$. Furthermore, we have that
\begin{align*}
\log \Biggl( \sum_{w \in B_k(X,H)}  e^{S_k f(w)} \Biggr) & = \log \Biggl( \sum_{j=1}^m  e^{S_k f(x^k_j)} \Biggr) \\
& = H_{\mu_k} ( \xi^k ) + \int f d\nu_k,
\end{align*}
where $\xi$ is the natural partition of $X$ according to the symbol in the zero coordinate, and $\xi^k$ is the $k$-fold join of $\xi$. Arguing as in Proposition 3.6 of \cite{LedWalters}, we obtain that
\begin{equation} \label{Eqn:Ramona}
 \limsup_{k \to \infty} \frac{1}{k} \log \Biggl( \sum_{w \in B_k(X,H)}  e^{S_k f(w)} \Biggr) \leq h(\nu) + \int f d\nu.
\end{equation}

Now we claim that $\nu$ is supported on $Y$. Let $[w]$ be a cylinder set in $X$ such that $Y \cap [w] = \varnothing$ and $w$ has length $N$. We show that $\nu([w])=0$. Since $Y \cap [w] = \varnothing$ and since $X$ is compact, there must exist $k_0$ such that for all $k \geq k_0$ and for all $u$ in $B_k(X,H)$, it holds that $w$ is not a subword of $u$. Then for $u$ in $B_k(X,H)$, $x$ in $[u]$, and $j = 0, \dots , k-N$, we have that $S^j(x) \notin [w]$. Hence $S^j \mu_k ( w ) = 0 $ for $j = 0, \dots, k-N$, and therefore
\begin{align*}
 \nu_k(w) = \frac{1}{k} \sum_{j=0}^{k-1} S^j \mu_k (w) = \frac{1}{k} \sum_{j=k-N+1}^{k-1} S^j \mu_k (w) \leq \frac{N}{k}.
\end{align*}
Letting $k$ tend to infinity along the subsequence $(k_j)$, we obtain that $\nu([w]) = 0$, as desired. Hence $\nu$ is supported on $Y$.

Then by (\ref{Eqn:Ramona}) and the variational principle for $P_Y(f)$, we have that
\begin{align*}
 \limsup_{k \to \infty} \frac{1}{k} \log \Biggl( \sum_{w \in B_k(X,H)}  e^{S_k f(w)} \Biggr) \leq h(\nu) + \int f d\nu \leq P_Y(f),
\end{align*}
which establishes (\ref{Eqn:Sophie}) and finishes the proof.
\end{proof}

\subsection{Proof of Theorem \ref{Thm:EscapeRate}}

Having established Proposition \ref{Prop:Crayola}, we are now in a position to prove Theorem \ref{Thm:EscapeRate}. The proof simply uses Proposition \ref{Prop:Crayola} to reduce Theorem \ref{Thm:EscapeRate} to Theorem \ref{Thm:Pressure}.

\vspace{2mm}

\begin{PfOfThmEscapeRate}

Let $X$ be a non-trivial mixing SFT, $f : X \to \mathbb{R}$ a H\"{o}lder continuous potential with associated Gibbs measure $\mu$, and $\gamma_0 = \gamma_0(X,f)$ as in Theorem \ref{Thm:Pressure}. Let $\alpha \in (\gamma_0,1]$. Let $\epsilon >0$. By Theorem \ref{Thm:Pressure}, there exists $\rho >0$ such that for all large enough $n$,
\begin{equation*}
\mathbb{P} \biggl( \bigl| P_{Y_n}(f) - ( P_X(f) + \log(\alpha) ) \bigr| \geq \epsilon \biggr) < e^{- \rho n}.
\end{equation*}
Observe that $Y_n$ is the survivor set of the open system $(X,\sigma,H_n)$. Then by Proposition \ref{Prop:Crayola}, we have
\begin{equation*}
- \varrho(\mu : H_n) = P_X(f) - P_{Y_n}(f). 
\end{equation*}
Then for all large enough $n$, we see that
\begin{align*}
 \mathbb{P} \biggl( \bigl| \varrho(\mu : H_n) - \log(\alpha) \bigr| \geq \epsilon \biggr) = \mathbb{P} \biggl(  \bigl| P_{Y_n}(f) - ( P_X(f) + \log(\alpha) ) \bigr|  \geq \epsilon \biggr) < e^{-\rho n},
\end{align*}
as was to be shown.
\end{PfOfThmEscapeRate}

\bibliographystyle{amsplain}
\bibliography{randomHolesRefs}

\end{document}